\documentclass[twoside, 11pt]{article}

\usepackage[top=50mm, bottom=50mm, left=50mm, right=50mm]{geometry}
\usepackage{lineno}
\usepackage{setspace}
\usepackage{amssymb}
\usepackage{amsmath}
\usepackage{amsthm}
\usepackage{amsfonts}
\usepackage{epsfig}
\usepackage{graphicx}
\usepackage{graphics}
\usepackage{float}
\usepackage{multirow}
\usepackage{lineno}
\usepackage{fullpage}
\usepackage[normalem]{ulem} 
\usepackage{makeidx}
\usepackage{xspace}
\usepackage{wrapfig}
\usepackage{color}
\usepackage{mathtools}

\makeindex

\theoremstyle{definition}
\newtheorem{theorem}{Theorem}
\newtheorem{Definition}[theorem]{Definition}
\newtheorem{corollary}[theorem]{Corollary}

\newtheorem{Lemma}[theorem]{Lemma}

\newtheorem*{rmk}{Remark}
\numberwithin{theorem}{section}

\newcommand{\Pro}{{\mathbf{P}}}

\newcommand{\Pas}{{\Pro\text{-a.s.}}}
\newcommand{\Fi}{{(\mathcal{F}_t)_{t\ge0}}}

\renewcommand{\labelenumi}{(\arabic{enumi})}

\allowdisplaybreaks[4]

\begin{document}

\title{The well-posedness of the stochastic nonlinear Schr$\ddot{\text{o}}$dinger equations in $H^2(\mathbb{R}^d)$ \footnote{ASM Subject Classifications: 60H15, 35B65, 35J10}}
 
\author{
Isamu D\^oku \\ Department of Mathematics, Faculty of Education, Saitama University \\ 255 Shimo-$\bar{\text{O}}$kubo, Sakura-ku, Saitama City 338-8570, Japan \and Shunya Hashimoto \\ Department of Mathematics, Faculty of Science, Saitama University \\ 255 Shimo-$\bar{\text{O}}$kubo, Sakura-ku, Saitama City 338-8570, Japan  \and Shuji Machihara \\ Department of Mathematics, Faculty of Science, Saitama University \\ 255 Shimo-$\bar{\text{O}}$kubo, Sakura-ku, Saitama City 338-8570, Japan
}

\date{}

\maketitle

\begin{abstract}
The Cauchy problem for the stochastic nonlinear Schr\"odinger equation with multiplicative noise is considered where the nonlinear term is of power type and the noise coeﬃcients are purely imaginary numbers. The main purpose of this paper is to construct classical solutions in $H^2(\mathbb{R}^d)$ for the problem. The techniques of Kato \cite{K87,K89} work well in overcoming smoothness problems even for the stochastic equations. 
\end{abstract}


\section{Introduction and main results}
\subsection{Stochastic and Deterministic nonlinear Schr\"odinger equations}
We consider the Cauchy problem for the stochastic nonlinear Schr\"odinger equation (SNLS) with
a multiplicative noise in the general spatial dimension $d\in\mathbb{N}$:
\begin{eqnarray}
\label{SNLS}
\begin{cases}
idX(t,\xi)=\Delta X(t,\xi)dt+\lambda|X(t,\xi)|^{\alpha-1}X(t,\xi)dt\\
\hspace{5em} -i\mu(\xi)X(t,\xi)dt+iX(t,\xi)dW(t,\xi), \quad &t\in(0,T), \ \xi\in \mathbb{R}^d, \\ 
X(0,\xi)=x(\xi), &\xi\in \mathbb{R}^d,
\end{cases}
\end{eqnarray}
where the constants $\lambda =\pm 1, \ \alpha>1$. The Wiener process $W$ and $\mu$ are 
the functions as follows, where  
\begin{align}
W(t,\xi)&=\sum_{j=1}^Ni\phi_j(\xi)\beta_j(t), \quad t\ge0, \ \xi\in \mathbb{R}^d, \label{def-W}\\
\mu(\xi)&=\frac{1}{2}\sum_{j=1}^N\phi_j^2(\xi), \quad \xi \in \mathbb{R}^d. \label{def-mu}
\end{align}
Here, $N\in \mathbb{N}\cup \{ +\infty \}$, and the 
$\phi_j(\xi)$ are real-valued functions.  
The $\beta_j(t)$ are real-valued independent Brownian motions with respect to a 
probability space $(\Omega,\mathcal{F},\Pro)$ with natural filtration 
$(\mathcal{F}_t)_{t\ge0}, \ 1\le j\le N$. This equation was introduced by \cite{BRZ14,BRZ16}. In this paper we assume $N<\infty$ as in \cite{BRZ14,BRZ16}. But our techniques easily go over to the case where $N=+\infty$ (i.e. infinite dimensional noise). We refer to \cite[Remark 2.3.13]{Z14} for details. 

There are a number of papers with results on the deterministic nonlinear Schr\"odinger equation which is given as in the case 
$W=0$ in \eqref{SNLS}, \eqref{def-W} and \eqref{def-mu}. 
We introduce the typical condition for the power of nonlinearity,
\begin{align}
\label{power-condition}
1<\alpha<1+\frac{4}{(d-2s)^+},
\end{align}
for the time local well-posedness in $H^s(\mathbb{R}^d), s\ge0$. Here,
\begin{align*}
\frac{1}{h^+}=
\begin{cases}
\frac{1}{h}, & h>0, \\
\infty, & h\le0.
\end{cases}
\end{align*}
This condition is derived from a scaling argument, and actually the time local well-posedness was 
shown for $s=1$ by Ginible-Velo \cite{GV79}, for $s=0$ by Tsutsumi \cite{T87}, for $s=2$ by Kato \cite{K87,K89} and for the general $0\le s<d/2$ 
by Cazenave-Weissler \cite{CW90}. We call the case $s=0$ charge class, $s=1$ energy class respectively since there are 
the conservation laws for $L^2(\mathbb{R}^d)$ and $H^1(\mathbb{R}^d)$ respectively. 
We sometimes call the case $s=2$ classical solution class since the nonlinear Schr\"odinger equation is a second-order partial differential equation with the Laplacian in $x$. 

So a natural question arises. Can we solve the stochastic nonlinear Schr\"odinger equation under 
the condition of \eqref{power-condition}?
There are some results. The conservative case 
(i.e. the special case of purely 
imaginary noise $\text{Re}W=0$) was studied in \cite{BD99,BD03}. In \cite{BD99} local existence and uniqueness of solutions in $L^2(\mathbb{R}^d)$ were proved for $\alpha$ satisfying
\begin{eqnarray}
\begin{cases}
1<\alpha <1+\frac{4}{d} & d=1,2, \\
1<\alpha <1+\frac{2}{d-1} & d\ge 3,
\end{cases}
\end{eqnarray}
and in \cite{BD03} local existence and uniqueness of solutions in $H^1(\mathbb{R}^d)$ were proved for $\alpha$ satisfying
\begin{eqnarray}
\begin{cases}
1<\alpha <\infty & d=1,2, \\
1<\alpha <5 & d=3, \\
1<\alpha <1+\frac{2}{d-1} \ \text{or} \ 2\le \alpha <1+\frac{4}{d-2} & d=4,5, \\
1<\alpha <1+\frac{2}{d-1} & d\ge 6.
\end{cases}
\end{eqnarray}

In the non-conservative case, V. Barbu, M. R$\ddot{\text{o}}$ckner and D. Zhang \cite{BRZ14, BRZ16} solved the time local well-posedness in $L^2({\mathbb{R}^d})$ 
and $H^1({\mathbb{R}^d})$ under the full condition \eqref{power-condition} 
by using the rescaling approach with an additional condition (see below (H1$)_s$).
This transformation reduces the stochastic nonlinear Schr\"odinger equation to an equivalent random Schr\"odinger equation. 

See also \cite{BRZ18, BHM20, BHW19, BHW22, HRZ19, H20, SZ23} for further work on the stochastic nonlinear Schr\"odinger equation.

In this paper, we study the well-posedness in $H^2(\mathbb{R}^d)$ in the conservation case by making use of the rescaling approach as a main tool for dealing with the multiplicative noise, where we need to take advantage of a slight modiﬁcation of the deterministic Strichartz estimates adapted to the setting of $H^2(\mathbb{R}^d)$. 
Unlike the $L^2$ and $H^1$-solutions, the difficulties in proving the $H^2$-solution are the smoothness argument for the function of the nonlinear terms and the modification of Kato's technique \cite{K87,K89} by the modified evolution operator.

\subsection{Results by the rescaling approach}
In this subsection 
we introduce the four results on the time local well-posedness for \eqref{SNLS} by using the rescaling approach.
The two results in Theorem \ref{Hslocal} for $s=0,1,$ are already known, and the other two results for $s=2$ with different settings in Theorems \ref{mainH2} and \ref{corH2} are the main results in this paper.
To state those results precisely, we introduce an assumption on $\{ \phi_j \}_{j=1}^N$ as follows. We assume some decay conditions.

(H1$)_s$ $\phi_j \in C^{\infty}_b(\mathbb{R}^d)$ such that
\[ \lim_{|\xi|\to \infty}\zeta(\xi)|\partial^{\gamma}\phi_j(\xi)|=0, \]
\qquad where $\gamma$ is a multi-index such that $|\gamma|\le s+2, \ 1\le j\le N$ and 
\begin{eqnarray*}
\zeta (\xi)=
\begin{cases}
1+|\xi|^2, & \text{if} \ d\not= 2; \\
(1+|\xi|^2)(\log (3+|\xi|^2))^2, & \text{if} \ d=2.
\end{cases}
\end{eqnarray*}
\begin{rmk}
We remark that the assumption (H1$)_s$ is the almost same with \cite{BRZ14,BRZ16}. We assume $|\gamma|\le 4$ on the derivative of $\phi_j$ though it was assumed $|\gamma|\le 2$ in \cite{BRZ14} and $|\gamma|\le 3$ in \cite{BRZ16}. 
\end{rmk}
We define the notion of solvability for the stochastic equation.
\begin{Definition}
\label{Xdef}
Let $x\in H^s \ (s=0,1,2)$ and let $\alpha>1$. Fix $0<T<\infty$. A solution of \textnormal{(\ref{SNLS})} is a pair $(X,\tau)$, where $\tau(\le T)$ is an $(\mathcal{F}_t)$-stopping time, and $X=(X(t))_{t\in [0,T]}$ is an $H^s$-valued continuous $(\mathcal{F}_t)$-adapted process, such that $|X|^{\alpha-1}X\in L^1(0,\tau;H^{s-2})$, $\Pas$, and it satisfies $\Pas$ 
\begin{align}
\label{WF}
X(t)=x-\int_0^{t\wedge\tau}(i\Delta X(r)+\mu X(r)+\lambda i|X(r)|^{\alpha-1}X(r))dr+\int_0^{t\wedge\tau}X(r)dW(r), \quad t\in [0,T],
\end{align}
as an equation in $H^{s-2}$.
\end{Definition}
We say that uniqueness for (\ref{SNLS}) holds
in the function space $S$, if for any two solutions of (\ref{SNLS}) $(X_i,\tau_i), \ X_i\in S, \ i=1,2$, it holds $\Pas$ that $X_1=X_2 \ \text{on} \ [0,\tau_1\wedge \tau_2]$.

We state the two known results for $s=0$ and $s=1$ 
which were shown in \cite[Theorem 2.2]{BRZ14} 
and \cite[Theorem 2.1]{BRZ16} respectively in the following theorem.
\begin{theorem}[$H^s$ local well-posedness for $s=0,1$]
\label{Hslocal}
\quad \\
Let $s=0$ or $1$. Assume (H1$)_s$. Let $\alpha$ satisfy (\ref{power-condition}). Then, for each $x\in H^s$ and $0<T<\infty$, there is a sequence of local solutions $(X_n,\tau_n)$ of \textnormal{(\ref{SNLS})}, $n\in \mathbb{N}$, where $\tau_n$ is a sequence of increasing stopping times. For every $n\ge 1$, it holds $\Pas$ that 
\begin{align*}
X_n|_{[0,\tau_n]}\in C([0,\tau_n];H^s)\cap L^{\gamma}(0,\tau_n;W^{s,\rho}),
\end{align*}
and uniqueness holds in the function space $C([0,\tau_n];H^s)\cap L^{\gamma}(0,\tau_n;W^{s,\rho})$ where $(\rho,\gamma)$ is any Strichartz pair. 
Here, the Strichartz pair for spatial dimension $d=1,2,3,\cdots,$ is defined by
\begin{eqnarray*}
(p_i,q_i)\in [2,\infty]\times [2,\infty]:\frac{2}{q_i}=\frac{d}{2}-\frac{d}{p_i}, \quad \text{if} \ d\not= 2, \\
(p_i,q_i)\in [2,\infty)\times (2,\infty]:\frac{2}{q_i}=\frac{d}{2}-\frac{d}{p_i}, \quad \text{if} \ d=2.
\end{eqnarray*}
Moreover, defining $\displaystyle \tau^*(x)=\lim_{n\to \infty}\tau_n$ and $\displaystyle X=\lim_{n\to \infty}X_n\mathbf{1}_{[0,\tau^*(x))},$ for $n\ge1$ and $\Pas \ \omega \in \Omega$, the map $x\to X(\cdot,x,\omega)$ is continuous from $H^s$ to $L^{\infty}(0,\tau_n;H^s)\cap L^{\gamma}(0,\tau_n;W^{s,\rho})$ where $(\rho,\gamma)$ is any Strichartz pair. Furthermore,  we have the blowup alternative, that is, for $\Pas \ \omega$, if $\tau_n(\omega)<\tau^*(x)(\omega), \ \forall n\in \mathbb{N}$, then 
\[ \lim_{t\to \tau^*(x)(\omega)}\|X(t)(\omega)\|_{H^s}=\infty. \]
\end{theorem}
Now we state our main theorems. The main theorem is shown in the range of the following exponents.
\begin{align}
\label{H2power}
\begin{cases}
1<\alpha<\infty, & 1\le d\le 4, \\
1<\alpha<1+\frac{2}{d-2} \ \text{or} \ 2\le \alpha<1+\frac{4}{d-4}, & 5\le d\le 7, \\
1<\alpha<1+\frac{2}{d-2}, & d\ge 8.
\end{cases}
\end{align}
\begin{theorem}[$H^2$ local well-posedness]
\label{mainH2}
\quad \\
Assume (H1$)_2$ and let $\alpha$ satisfy (\ref{H2power}). Also, assume that $\phi_j$ is a first-order polynomial for any $j=1,2,\cdots,N$. Then, for each $x\in H^2$ and $0<T<\infty$, there is a sequence of local solutions $(X_n,\tau_n)$ of \textnormal{(\ref{SNLS})}, $n\in \mathbb{N}$, where $\tau_n$ is a sequence of increasing stopping times. For every $n\ge 1$, it holds $\Pas$ that 
\begin{align*}
X_n|_{[0,\tau_n]}\in C([0,\tau_n];H^2),
\end{align*}
and uniqueness holds in the function space $C([0,\tau_n];H^2)$. Moreover, defining $\displaystyle \tau^*(x)=\lim_{n\to \infty}\tau_n$ and $\displaystyle X=\lim_{n\to \infty}X_n\mathbf{1}_{[0,\tau^*(x))},$ for $n\ge1, \ \Pas \ \omega \in \Omega$ and $0\le s<2$, the map $x\to X(\cdot,x,\omega)$ is continuous from $H^2$ to $L^{\infty}(0,\tau_n;H^s)$.
\end{theorem}
\begin{rmk}
The proof of Theorem \ref{mainH2} is based on Kato's technique \cite{K87,K89}.
Unlike the $H^2$-solution of the deterministic Schr\"odinger equation, the evolution operator modified by the rescaling approach is not commutative with $\nabla$ and $\Delta$ and is not a semigroup.
This causes the gap of exponent $\alpha$ in \eqref{H2power}.
This gap came from the condition of Sobolev's embedding.
And, we note that we do not use Kato's technique when the exponent $\alpha$ is greater than or equal to 2.
\end{rmk}
If we assume that the exponent $\alpha$ is greater than or equal to 2, not only can the assumption about $\phi_j$ be removed, but it also improves the properties of the solution as follows.
We can show strictly the space to which the solution belongs and the continuous dependence of the initial data.
\begin{theorem}[$H^2$ local well-posedness for sufficiently smooth nonlinearities]
\label{corH2}
\quad \\
Assume $d\le 7$ and (H1$)_2$. Let $\alpha$ satisfy $2\le \alpha <1+\frac{4}{(d-4)^+}.$ 
Then, for each $x\in H^2$ and $0<T<\infty$, there is a sequence of local solutions $(X_n,\tau_n)$ of \textnormal{(\ref{SNLS})}, $n\in \mathbb{N}$, where $\tau_n$ is a sequence of increasing stopping times. For every $n\ge 1$, it holds $\Pas$ that 
\begin{align*}
X_n|_{[0,\tau_n]}\in C([0,\tau_n];H^2)\cap L^q(0,\tau_n;W^{2,p}),
\end{align*}
and uniqueness holds in the function space 
$C([0,\tau_n];H^2)$ where $(p,q)$ is given by
\begin{align*}
\begin{cases}
\text{any Strichartz pair} & d=1,2,3, \\
(p,q)=(\frac{2(\alpha+2)}{\alpha+1},\alpha+2) & d=4, \\
(p,q)=(\frac{d(\alpha+1)}{d+2\alpha-2},\frac{4(\alpha+1)}{(d-4)(\alpha-1)}) & d=5,6,7.
\end{cases}
\end{align*}
Moreover, defining $\displaystyle \tau^*(x)=\lim_{n\to \infty}\tau_n$ and $\displaystyle X=\lim_{n\to \infty}X_n\mathbf{1}_{[0,\tau^*(x))},$ for $n\ge1$ and $\Pas \ \omega \in \Omega$, the map $x\to X(\cdot,x,\omega)$ is continuous from $H^2$ to $L^{\infty}(0,\tau_n;H^2)\cap L^q(0,\tau_n;W^{2,p})$. Furthermore, we have the blowup alternative, that is, for $\Pas \ \omega$, if $\tau_n(\omega)<\tau^*(x)(\omega),$ \ for every $n\in \mathbb{N}$, then 
\[ \lim_{t\to \tau^*(x)(\omega)}\|X(t)(\omega)\|_{H^2}=\infty. \]
\end{theorem}

It is natural to consider the existence of global solutions. We introduce the known result in $H^1$ which is [8, Theorem 1.2].
\begin{theorem}[$H^1$ global well-posedness]
\quad \\
Assume (H1$)_1$. Let $\alpha$ satisfy $1< \alpha <1+\frac{4}{(d-2)^+} \ \text{if} \ \lambda=-1,$ or 
$1< \alpha <1+\frac{4}{d} \ \text{if} \ \lambda=1.$ Then, for each $x\in H^1$ and $0<T<\infty$, there exists a unique $H^1$-global solution $(X,T)$ of (\ref{SNLS}), such that 
\begin{align*}
X\in L^{\gamma}(0,T;W^{1,\rho}), \ \Pas
\end{align*}
where $(\rho,\gamma)$ is any Strichartz pair. Moreover, for $\Pas \ \omega,$ the map $x\to X(\cdot,x,\omega)$ is continuous from $H^1$ to $L^{\infty}(0,T;H^1)\cap L^{\gamma}(0,T;W^{1,\rho})$ where $(\rho,\gamma)$ is any Strichartz pair.
\end{theorem}
The following theorem which corresponds to $H^2$-solutions is new.
\begin{theorem}[The existence of global solutions in $H^2$]
\label{H2glob}
\quad \\
Assume $d\le7$ and (H1$)_2$. Let $\alpha$ satisfy $2\le \alpha <1+\frac{4}{(d-4)^+}.$
In addition, we assume that for $x\in H^2, \ 0<T<\infty,$ the following holds.
\[ \|X\|_{L^{\infty}(0,T;H^1)}+\|X\|_{L^q(0,T;W^{1,p})}<\infty, \ \Pas \]
where $(p,q)=(\alpha+1,\frac{4(\alpha+1)}{d(\alpha-1)})$. Then, there exists a unique $H^2$-global solution $(X,T)$ of (\ref{SNLS}).
\end{theorem}
So by restricting the range of power to that of $H^1$ well-posedness, we may apply Theorem \ref{H2glob} to have the following $H^2$ globally well-posed result, as a corollary.
This method is sometimes called persistence of regularity.
\begin{corollary}
Assume $d\le 7$ and (H1$)_2$. Let $\alpha$ satisfy $2\le \alpha <1+\frac{4}{(d-2)^+} \ \text{if} \ \lambda=-1,$ or 
$2\le \alpha <1+\frac{4}{d} \ \text{if} \ \lambda=1.$ Then, there exists a unique $H^2$-global solution $(X,T)$ of (\ref{SNLS}).
\end{corollary}

\section{Rescaling approach}
\setcounter{equation}{0}

The main tool to prove Theorem \ref{mainH2} and Theorem \ref{corH2}
  is based on the rescaling approach as used for Theorem \ref{Hslocal} in \cite{BRZ14,BRZ16}. 
We apply the rescaling transformation
\begin{align}
\label{RT}
X(t,\xi)=e^{W(t,\xi)}y(t,\xi).
\end{align}
By an
application of It\^o's product formula, we see that $\Pas$
\[ dX=e^Wdy+e^WydW-\mu e^Wydt. \]
We apply (\ref{RT}) to (\ref{SNLS}) to have
\begin{eqnarray}
\label{RSNLS}
\begin{cases}
\displaystyle \frac{\partial y(t,\xi)}{\partial t}=A(t)y(t,\xi)-\lambda i|y(t,\xi)|^{\alpha-1}y(t,\xi), \\
y(0,\xi)=x(\xi),
\end{cases}
\end{eqnarray}
where
\begin{align}
\label{EO}
A(t)y(t,\xi)&=-ie^{-W}\Delta (e^Wy) \nonumber \\
&=-i(\Delta +b(t,\xi)\cdot \nabla +c(t,\xi))y(t,\xi), \\
\label{b}
b(t,\xi)&=2\nabla W(t,\xi), \\
\label{c}
c(t,\xi)&=\sum_{j=1}^d(\partial_jW(t,\xi))^2+\Delta W(t,\xi).
\end{align}
The definition of solutions to (\ref{RSNLS}) are given in the following sense (similar to Definition \ref{Xdef}).
\begin{Definition}
\label{ydef}
Let $x\in H^s \ (s=0,1,2)$ and let $\alpha>1$. Fix $0<T<\infty$. A $H^s$-solution of \textnormal{(\ref{RSNLS})} is a pair $(y,\tau)$, where $\tau(\le T)$ is an $(\mathcal{F}_t)$-stopping time, and $y=(y(t))_{t\in [0,T]}$ is an $H^s$-valued continuous $(\mathcal{F}_t)$-adapted process, such that $|y|^{\alpha-1}y\in L^1(0,\tau;H^{s-2})$, $\Pas$, and it satisfies $\Pas$
\begin{eqnarray}
\label{RWF}
y(t)=x+\int_0^{t\wedge\tau}A(r)y(r)dr-\int_0^{t\wedge\tau}\lambda i|y(r)|^{\alpha-1}y(r)dr, \quad t\in [0,T],
\end{eqnarray}
as an equation in $H^{s-2}$.
\end{Definition}
We say that uniqueness holds for (\ref{RSNLS}) in the function space $S$, if for any two solutions of (\ref{RSNLS}) $(y_i,\tau_i), \ y_i\in S, \ i=1,2$, it holds $\Pas$ that $y_1=y_2 \ \text{on} \ [0,\tau_1\wedge \tau_2]$.

The following theorem establishes the equivalence between 
the two definitions of solutions to (\ref{SNLS}) and (\ref{RSNLS}) respectively.
\begin{theorem}
\label{Xy}
For $s=0,1,2,$ the following holds.
\begin{enumerate}
\item Let $(y,\tau)$ be a $H^s$-solution of (\ref{RSNLS}) in the sense of Definition \ref{ydef}. Set $X:=e^Wy$. Then $(X,\tau)$ is a $H^s$-solution of (\ref{SNLS}) in the sense of Definition \ref{Xdef}.
\item Let $(X,\tau)$ be a $H^s$-solution of (\ref{SNLS}) in the sense of Definition \ref{Xdef}. 
Set $y:=e^{-W}X$. Then $(y,\tau)$ is a $H^s$-solution of (\ref{RSNLS}) in the sense of Definition \ref{ydef}.
\end{enumerate}
\end{theorem}
\begin{proof}
The cases of $L^2$ and $H^1$ were proved in \cite{BRZ14,BRZ16}. Therefore, we prove the case of $H^2$.
In the first case (1), since $x\in H^2\subset H^1$ and $y$ satisfies (\ref{RWF}) in $L^2\subset H^{-1}$, Lemma 2.4 in \cite{BRZ16} implies that $(X,\tau)$ is a solution of (\ref{SNLS}) in the sense of Definition \ref{Xdef}, in particular, $X$ solves (\ref{WF}) in $H^{-1}$. But, as $y\in C([0,T];H^2)$ and $e^W\in C([0,T];W^{2,\infty})$, we deduce that $X\in C([0,T];H^2)$. Hence, the right hand side of (\ref{WF}) is in $L^2$, which implies that $(X,\tau)$ is a solution of (\ref{SNLS}) in the sense of Definition \ref{Xdef}, thereby completing the proof of (1). The proof for (2) follows analogously.
\end{proof}
By the equivalence of two expressions of solutions via the rescaling transformation (\ref{RT}),
Theorems \ref{mainH2}, \ref{corH2} and \ref{H2glob} are reformulated below in Theorems \ref{ymain}, \ref{ycor} and \ref{yglob}, respectively.
\begin{theorem}
\label{ymain}
Assume (H1$)_2$. Let $\alpha$ satisfy (\ref{H2power}). Also, assume that $\phi_j$ is a first-order polynomial for any $j=1,2,\cdots,N$. Then, for each $x\in H^2$ and $0<T<\infty$, there is a sequence of local solutions $(y_n,\tau_n)$ of \textnormal{(\ref{RSNLS})}, $n\in \mathbb{N}$, where $\tau_n$ is a sequence of increasing stopping times. For every $n\ge 1$, it holds $\Pas$ that
\begin{align*}
y_n|_{[0,\tau_n]}\in C([0,\tau_n];H^2),
\end{align*}
and uniqueness holds in the function space $C([0,\tau_n];H^2)$. Moreover, defining $\displaystyle \tau^*(x)=\lim_{n\to \infty}\tau_n$ and $\displaystyle y=\lim_{n\to \infty}y_n\mathbf{1}_{[0,\tau^*(x))},$ for $n\ge1, \ \Pas \ \omega \in \Omega$ and $0\le s<2$, the map $x\to y(\cdot,x,\omega)$ is continuous from $H^2$ to $L^{\infty}(0,\tau_n;H^s)$.
\end{theorem}
\begin{theorem}
\label{ycor}
Assume $d\le 7$ and (H1$)_2$. Let $\alpha$ satisfy $2\le \alpha <1+\frac{4}{(d-4)^+}.$ Then, for each $x\in H^2$ and $0<T<\infty$, there is a sequence of local solutions $(y_n,\tau_n)$ of \textnormal{(\ref{RSNLS})}, $n\in \mathbb{N}$, where $\tau_n$ is a sequence of increasing stopping times. For every $n\ge 1$, it holds $\Pas$ that
\begin{align}
y_n|_{[0,\tau_n]}\in C([0,\tau_n];H^2)\cap L^{\gamma}(0,\tau_n;W^{2,\rho}),
\end{align}
and uniqueness holds in the function space $C([0,\tau_n];H^2)$ where $(\rho,\gamma)$ is any Strichartz pair. Moreover, defining $\displaystyle \tau^*(x)=\lim_{n\to \infty}\tau_n$ and $\displaystyle y=\lim_{n\to \infty}y_n\mathbf{1}_{[0,\tau^*(x))},$ for $n\ge1$ and $\Pas \ \omega \in \Omega$, the map $x\to y(\cdot,x,\omega)$ is continuous from $H^2$ to $L^{\infty}(0,\tau_n;H^2)\cap L^{\gamma}(0,\tau_n;W^{2,\rho})$. Furthermore, we have the blowup alternative, that is, for $\Pas \ \omega$, if $\tau_n(\omega)<\tau^*(x)(\omega),$ \ for every $n\in \mathbb{N}$, then
\[ \lim_{t\to \tau^*(x)(\omega)}\|y(t)(\omega)\|_{H^2}=\infty. \]  
\end{theorem}
\begin{theorem}
\label{yglob}
Assume $d\le 7$ and (H1$)_2$. Let $\alpha$ satisfy $2\le \alpha <1+\frac{4}{(d-4)^+}.$
In addition, we assume that for $x\in H^2, \ 0<T<\infty,$ the following holds.
\[ \|y\|_{L^{\infty}(0,T;H^1)}+\|y\|_{L^q(0,T;W^{1,p})}<\infty, \ \Pas \]
where $(p,q)=(\alpha+1,\frac{4(\alpha+1)}{d(\alpha-1)})$. Then, there exists a unique $H^2$-global solution $(y,T)$ of (\ref{RSNLS}).
\end{theorem}

\section{Deterministic Strichartz estimate}
\setcounter{equation}{0}

For the proof of  Theorem \ref{ymain} and Theorem \ref{ycor}, 
we discuss
the evolution operators and the corresponding Strichartz estimate which was shown in \cite{BRZ14} and \cite{BRZ16}. 
\begin{Lemma}
\label{evo}
Assume (H1$)_2$. For $\Pas \ \omega,$ the operator $A(t)$ defined in (\ref{EO}) generates evolution operators $U(t,s)=U(t,s,\omega), \ 0\le s\le t\le T,$ in the spaces $H^2(\mathbb{R}^d)$. Moreover, for each $x\in H^2(\mathbb{R}^d),$ the 
process $[s,T]\ni t \mapsto U(t,s)x\in H^2(\mathbb{R}^d)$
is continuous and $(\mathcal{F}_t)$-adapted, hence progressively measurable with respect to the filtration $(\mathcal{F}_t)_{t\ge s}$.
\end{Lemma}
\begin{proof}
The existence of the evolution operator $U$ generated by $A(t)$ is a direct consequence
of the fact that, for ($\Pas$) every $\omega \in \Omega$, the Cauchy problem 
\begin{align*}
\begin{dcases}
\frac{dy}{dt}=A(t)y, \\
y(s)=x, \quad s\le t<\infty,
\end{dcases}
\end{align*}
for each $x\in H^2(\mathbb{R}^d)$ has a unique solution $y \in C([s,T];H^2(\mathbb{R}^d))$ for all $T>s$. 

Indeed, by Theorem 1.1 in \cite{D96}, under our assumptions on $c$ and $b$, for each $x \in H^2$ and $f \in L^1(s,T;H^2)$, the Cauchy problem
\begin{align*}
\begin{dcases}
i\frac{\partial u}{\partial t}&=\Delta u+cu+b\cdot \nabla u+f \quad \text{in} \ (s,T)\times \mathbb{R}^d, \\
u(s)&=x,
\end{dcases}
\end{align*}
has a unique solution $u \in C([s,T];H^2)$. Therefore, an evolution operator  $U(t,s) \in L(H^2,H^2)$ 
is defined by $U(t,s)x=y(t), \ 0\le s\le t\le T$. For details, see Lemma 3.3 in \cite{BRZ14} (see also \cite{Z14}).
\end{proof}
We give a precise definition of local smoothing spaces introduced in \cite{MMT08} as used in this paper.
\begin{Definition}
\label{locsp}
Set $B_0=\{ |\xi|\le 2\}, \ B_j=\{ 2^j\le |\xi|\le 2^{j+1}\}, \ j\ge1,$ and $B_{<j}=\{ |\xi|\le 2^j\}.$ Let $A_j=[0, T]\times B_j, \ j\ge 0, \ A_{<j}=[0, T]\times B_{<j}, \ j\ge 1.$ We consider a dyadic partition of unity of frequency, i.e. $\displaystyle 1=\sum_{k=-\infty}^{\infty}S_k(D).$ We say a function $f$ is localized at frequency $2^k$, if $\hat{f}$ is supported in $\{ 2^{k-1}<|\xi|<2^{k+1}\}.$ The functions localized at frequency $2^k$ are measured using the norm
\begin{align*}
\|u\|_{X_k(T)}&=\|u\|_{L^2(A_0)}+\sup_{j>0}\|\langle \xi \rangle^{-\frac{1}{2}}u\|_{L^2(A_j)}, \ k\ge 0, \\
\|u\|_{X_k(T)}&=2^{\frac{k}{2}}\|u\|_{L^2(A_{<-k})}+\sup_{j\ge -k}\|(|\xi|+2^{-k})^{-\frac{1}{2}}u\|_{L^2(A_j)}, \ k<0,
\end{align*}
where $\langle \xi \rangle =\sqrt{1+|\xi|^2}.$ Then the local smoothing space $\widetilde{X}_{[0,T]}$ is defined by the norm 
\begin{align*}
\|u\|^2_{\widetilde{X}_{[0,T]}}&=\|\langle \xi \rangle^{-1}u\|^2_{L^2([0,T]\times \mathbb{R}^d)}+\sum_{k=-\infty}^{\infty}2^k\|S_ku\|^2_{X_k(T)}, \ d\not=2, \\
\|u\|^2_{\widetilde{X}_{[0,T]}}&=\|\langle \xi \rangle^{-1}(\log (2+|\xi|))^{-1}u\|^2_{L^2([0,T]\times \mathbb{R}^d)}+\sum_{k=-\infty}^{\infty}2^k\|S_ku\|^2_{X_k(T)}, \ d=2.
\end{align*}
\end{Definition}
We modify the deterministic Strichartz estimate in order to deal with $H^2$ solution.
\begin{Lemma}
\label{H2SE}
(Deterministic Strichartz estimate) \\
Assume (H1$)_2$. Then for any $T>0, \ u_0\in H^2$ and $f\in L^{q'_2}(0,T;W^{2,p'_2}),$ the solution of
\begin{align}
\label{uH^2}
u(t)=U(t,0)u_0+\int_0^tU(t,s)f(s)ds, \quad 0\le t\le T,
\end{align}
satisfies the estimates
\begin{align}
\label{SEL^2}
\|u\|_{L^{q_1}(0,T;L^{p_1})}&\le C_T(\|u_0\|_{L^2}+\|f\|_{L^{q'_2}(0,T;L^{p'_2})}), \\
\label{SEH^1}
\|u\|_{L^{q_1}(0,T;W^{1,p_1})}&\le C_T(\|u_0\|_{H^1}+\|f\|_{L^{q'_2}(0,T;W^{1,p'_2})}),
\end{align}
and
\begin{align}
\label{SEH^2}
\|u\|_{L^{q_1}(0,T;W^{2,p_1})}\le C_T(\|u_0\|_{H^2}+\|f\|_{L^{q'_2}(0,T;W^{2,p'_2})}),
\end{align}
where $(p_1,q_1)$ and $(p_2,q_2)$ are Strichartz pairs, namely
\begin{eqnarray*}
(p_i,q_i)\in [2,\infty]\times [2,\infty]:\frac{2}{q_i}=\frac{d}{2}-\frac{d}{p_i}, \quad \text{if} \ d\not= 2, \\
(p_i,q_i)\in [2,\infty)\times (2,\infty]:\frac{2}{q_i}=\frac{d}{2}-\frac{d}{p_i}, \quad \text{if} \ d=2.
\end{eqnarray*}
Furthermore, the process $C_t, \ t\ge 0,$ can be taken to be $(\mathcal{F}_t)$-progressively measurable, increasing and continuous.
\end{Lemma}
\begin{proof}
The key estimates of the proof are the following known results, for $T>0$,
\begin{align}
\label{XEL^2}
\|u\|_{L^{q_1}(0,T;L^{p_1})\cap \widetilde{X}_{[0,T]}}&\le C_T(\|u_0\|_{L^2}+\|f\|_{L^{q'_2}(0,T;L^{p'_2})+\widetilde{X}'[0,T]}), \\
\label{XEH^1}
\|\nabla u\|_{L^{q_1}(0,T;L^{p_1})\cap \widetilde{X}_{[0,T]}}&\le C_T(\|u_0\|_{H^1}+\|f\|_{L^{q'_2}(0,T;W^{1,p'_2})+\widetilde{X}'[0,T]}),
\end{align}
where $\widetilde{X}'$ is dual space of $\widetilde{X}$.
Indeed, (\ref{XEL^2}) is shown in \cite{BRZ14} and (\ref{XEH^1}) is shown in \cite{BRZ16} by using (\ref{XEL^2}). (\ref{SEL^2}) and (\ref{SEH^1}) are direct conclusion from (\ref{XEL^2}) and (\ref{XEH^1}) respectively. So we give a proof of (\ref{SEH^2}) only. The proof 
is based on Theorem 1.13 and Proposition 2.3(a) in \cite{MMT08}. Let us use the notation $D_t:=-i\partial_t, \ D_j:=-i\partial_j, \ 1\le j\le d,$ to rewrite (\ref{uH^2}) in the form 
\[ D_tu=(D_ja^{jk}D_k+D_j\widetilde{b}^j+\widetilde{b}^jD_j+\widetilde{c})u-if, \]
with $a^{jk}=\delta_{jk}, \ \widetilde{b}^j=-i\partial_jW_t$ and $\displaystyle \widetilde{c}=-\sum_{j=1}^d(\partial_jW)^2, \ 1\le j,k\le d$.

Direct computations show
\begin{align}
\label{Delta}
D_t\Delta u =&(-\Delta +D_j\widetilde{b}^j+\widetilde{b}^jD_j+\widetilde{c})\Delta u+2(D_j\nabla \widetilde{b}^j+\nabla \widetilde{b}^jD_j+\nabla \widetilde{c})\nabla u \nonumber \\
&+(D_j\Delta \widetilde{b}^j+\Delta \widetilde{b}^jD_j+\Delta \widetilde{c})u-i\Delta f.
\end{align}
We regard (\ref{Delta}) as the equation for the unknown $\Delta u$ and treat the lower order term $(D_j\nabla \widetilde{b}^j+\nabla \widetilde{b}^jD_j+\nabla \widetilde{c})\nabla u$ and $(D_j\Delta \widetilde{b}^j+\Delta \widetilde{b}^jD_j+\Delta \widetilde{c})u$ as equal terms with $\Delta f$. This leads to
\begin{align}
\label{deltau}
\Delta u(t)=U(t,0)\Delta u_0 &+\int_0^tU(t,s)[2i(D_j\nabla \tilde{b}^j(s)+\nabla \tilde{b}^j(s)D_j+\nabla \tilde{c}(s))\nabla u \nonumber \\
&+i(D_j\Delta \tilde{b}^j(s)+\Delta \tilde{b}^j(s)D_j+\Delta \tilde{c}(s))u+\Delta f(s)]ds.
\end{align}
Hence applying (\ref{XEL^2}),(\ref{XEH^1}) to (\ref{deltau}) and then using Theorem 1.13 and Proposition 2.3(a) in \cite{MMT08} to control the lower order term, we derive that 
\begin{align}
\label{lowerest}
&\|\Delta u\|_{L^{q_1}(0,T;L^{p_1})\cap \tilde{X}_{[0,T]}} \nonumber \\
&\le C_T[ \|\Delta u_0\|_{L^2}+\|2i(D_j\nabla \tilde{b}^j+\nabla \tilde{b}^jD_j+\nabla \tilde{c})\nabla u \nonumber \\
&\hspace{10em} +i(D_j\Delta \tilde{b}^j+\Delta \tilde{b}^jD_j+\Delta \tilde{c})u+\Delta f\|_{L^{q'_2}(0,T;L^{p'_2})+\tilde{X'}_{[0,T]}}] \nonumber \\
&\le C_T[ \|\Delta u_0\|_{L^2}+\|2i(D_j\nabla \tilde{b}^j+\nabla \tilde{b}^jD_j+\nabla \tilde{c})\nabla u\|_{\tilde{X'}_{[0,T]}} \nonumber \\
&\hspace{10em} +\|i(D_j\Delta \tilde{b}^j+\Delta \tilde{b}^jD_j+\Delta \tilde{c})u\|_{\tilde{X'}_{[0,T]}}+\|\Delta f\|_{L^{q'_2}(0,T;L^{p'_2})}] \nonumber \\
&\le C_T\left[ \|\Delta u_0\|_{L^2}+\kappa^1_T\|\nabla u\|_{\tilde{X}_{[0,T]}}+\kappa^2_T\|u\|_{\tilde{X}_{[0,T]}}+\|\Delta f\|_{L^{q'_2}(0,T;L^{p'_2})} \right] \\
&\le C_T[\|\Delta u_0\|_{L^2}+C_T\kappa^1_T(\|u_0\|_{H^1}+\|f\|_{L^{q'_2}(0,T;W^{1,p'_2})}) \nonumber \\
&\hspace{10em} +C_T\kappa^2_T(\|u_0\|_{L^2}+\|f\|_{L^{q'_2}(0,T;L^{p'_2})})+\|\Delta f\|_{L^{q'_2}(0,T;L^{p'_2})}] \nonumber \\
&=C_T(C_T\kappa^1_T+C_T\kappa^2_T+1)\left[ \|u_0\|_{H^2}+\|f\|_{L^{q'_2}(0,T;W^{2,p'_2})} \right], \nonumber 
\end{align}
here we are faced at the term of the third derivative of $\tilde{b}^j$ which corresponds to the 4th derivative of $W$ 
which means the 4th derivative of $\phi_j$.
This together with (\ref{SEL^2}),(\ref{SEH^1}) yields the estimate (\ref{SEH^2}). 
The $H^2$ continuity follows from Strichartz estimate (\ref{SEH^2}) in the usual way.
Now, we set
\begin{align}
C_t&=\sup \{\|U(\cdot,0)u_0\|_{L^{q_1}(0,t;W^{2,p_1})};\|u_0\|_{H^2}\le 1 \} \nonumber \\
& \quad +\sup \left\{ \left\| \int_0^{\cdot}U(\cdot,s)f(s)ds \right\|_{L^{q_1}(0,t;W^{2,p_1})};||f||_{L^{q'_2}(0,t;W^{2,p'_2})}=1 \right\}.
\end{align}
Then it is analogous to the proof of Lemma 4.1 in \cite{BRZ14} that the stochastic process $C_t, \ t\ge 0$ is $(\mathcal{F}_t)$-progressively measurable, increasing, and continuous (see also \cite{Z14}).
\end{proof}
\begin{rmk}
The estimate \eqref{XEL^2} holds only on the bounded interval $[0,T]$. See \cite[Appendix]{BRZ14} for details. Therefore, the main result statements are that well-posedness is obtained whenever $T>0$ is fixed.
\end{rmk}

\section{Proof of Theorem \ref{ymain} and Theorem \ref{ycor}}
\setcounter{equation}{0}

\begin{rmk}
The proof of Theorem \ref{ymain} is shown for the case $1<\alpha\le 1+\frac{2}{(d-2)^+}$ only.
For $3\le d\le 7$, Theorem \ref{ycor} yields Theorem \ref{ymain}.
\end{rmk}
\begin{proof}[Proof of Theorem \ref{ymain}]
Set $g(y)=|y|^{\alpha-1}y$. We solve the weak equation (\ref{RWF}) in the mild sense, namely 
\begin{align}
y(t)=U(t,0)x-\lambda i\int_0^tU(t,s)g(y(s))ds.
\end{align}
We consider the following map   
\[ F(y)(t)=U(t,0)x-\lambda i\int_0^tU(t,s)g(y(s))ds. \]

The local solutions $\{ (y_n,\tau_n) \}_{n\ge 1}$ of (\ref{RSNLS}) will be constructed explicitly below 

\textbf{Step 1.} First, we find $(y_1,\tau_1)$. 
Choose the Strichartz pair $(p,q)=(\frac{4\alpha}{\alpha+1},\frac{8\alpha}{d(\alpha-1)})$.
Fix $\omega \in \Omega$ and consider $F$ on the set
\[ \mathcal{Y}^{\tau_1}_{M_1}=\{ y\in L^{\infty}(0,\tau_1;L^2);\|y\|_{str(\tau_1)}+\|\partial_ty\|_{str(\tau_1)}+\|\Delta y\|_{L^{\infty}(0,\tau_1;L^2)}\le M_1 \}, \]
where $\|y\|_{str(t)}:=\|y\|_{L^{\infty}(0,t;L^2)}+\|y\|_{L^{q}(0,t;L^{p})}, \ \tau_1=\tau_1(\omega)\in (0,T]$ and $M_1=M_1(\omega)>0$ are random variables. The distance is defined by $d(y_1,y_2)=\|y_1-y_2\|_{str}$. 
We differentiate with respect to $t$,
\begin{align}
\label{partialt1}
\partial_t \left( \int_0^tU(t,s)g(y(s))ds \right)&=\int_0^t\partial_t U(t,s)g(y(s))ds+U(t,t)g(y(t)) \nonumber \\
&=\int_0^tA(t)U(t,s)g(y(s))ds+g(y(t)).
\end{align}
Also, we change variables
\[ \int_0^tU(t,s)g(y(s))ds=\int_0^tU(t,t-s)g(y(t-s))ds, \]
and from
\begin{align*}
\partial_t U(t,s)x&=A(t)U(t,s)x, \ t\ge s, \\
\partial_s U(t,s)x&=-U(t,s)A(s)x, \ t\ge s,
\end{align*}
we have
\begin{align}
\label{partialt2}
& \partial_t \left( \int_0^tU(t,s)g(y(s))ds \right)=\partial_t \left( \int_0^tU(t,t-s)g(y(t-s))ds \right) \nonumber \\
&=\int_0^t\partial_t\left( U(t,t-s)g(y(t-s)) \right) ds+U(t,0)g(y(0)) \\
&=\int_0^t\{A(t)U(t,s)-U(t,s)A(s)\}g(y(s))ds+\int_0^tU(t,s)(\partial_tg)(y(s))ds \nonumber \\
&\quad +U(t,0)g(x) \nonumber \\
&=:I+\int_0^tU(t,s)(\partial_tg)(y(s))ds+U(t,0)g(x). \nonumber 
\end{align}
We consider $I$. By (\ref{deltau}), we have
\begin{align*}
& \{A(t)U(t,s)-U(t,s)A(s)\}g(y(s)) \\
=&\{-i(\Delta+b(t)\cdot \nabla+c(t))U(t,s)+iU(t,s)(\Delta+b(s)\cdot\nabla+c(s))\}g(y(s)) \\
=&-i(\Delta U(t,s)-U(t,s)\Delta)g(y(s)) \\
&+\{-i(b(t)\cdot \nabla+c(t))U(t,s)+iU(t,s)(b(s)\cdot\nabla+c(s))\}g(y(s)) \\
=&-i\int_s^tU(t,\tau)[2i(D_j\nabla\widetilde{b}^j(\tau)+\nabla\widetilde{b}^j(\tau)D_j+\nabla\widetilde{c}(\tau))\nabla(U(\tau,s)g(y(s))) \\
& \qquad \qquad \qquad \qquad  +i(D_j\Delta\widetilde{b}^j(\tau)+\Delta\widetilde{b}^j(\tau)D_j+\Delta\widetilde{c}(\tau))U(\tau,s)g(y(s))]d\tau \\
&+\{-i(b(t)\cdot \nabla+c(t))U(t,s)+iU(t,s)(b(s)\cdot\nabla+c(s))\}g(y(s)).
\end{align*}
Therefore, we get
\begin{align*}
I=&-i\int_0^t\int_s^tU(t,\tau)[2i(D_j\nabla\widetilde{b}^j(\tau)+\nabla\widetilde{b}^j(\tau)D_j+\nabla\widetilde{c}(\tau))\nabla(U(\tau,s)g(y(s))) \\
& \qquad \qquad \qquad \qquad  +i(D_j\Delta\widetilde{b}^j(\tau)+\Delta\widetilde{b}^j(\tau)D_j+\Delta\widetilde{c}(\tau))U(\tau,s)g(y(s))]d\tau ds \\
&-i\int_0^tb(t)\cdot\nabla U(t,s)g(y(s))ds-i\int_0^tc(t)U(t,s)g(y(s))ds \\
&+i\int_0^tU(t,s)b(s)\cdot \nabla g(y(s))ds+i\int_0^tU(t,s)c(s)g(y(s))ds.
\end{align*}
Therefore, put this into (\ref{partialt2}) to have
\begin{align*}
& \partial_t \left( \int_0^tU(t,s)g(y(s))ds \right) \\
= &-i\int_0^t\int_s^tU(t,\tau)[2i(D_j\nabla\widetilde{b}^j(\tau)+\nabla\widetilde{b}^j(\tau)D_j+\nabla\widetilde{c}(\tau))\nabla(U(\tau,s)g(y(s))) \\
& \qquad \qquad \qquad \qquad  +i(D_j\Delta\widetilde{b}^j(\tau)+\Delta\widetilde{b}^j(\tau)D_j+\Delta\widetilde{c}(\tau))U(\tau,s)g(y(s))]d\tau ds \\
&-i\int_0^tb(t)\cdot\nabla U(t,s)g(y(s))ds-i\int_0^tc(t)U(t,s)g(y(s))ds \\
&+i\int_0^tU(t,s)b(s)\cdot \nabla g(y(s))ds+i\int_0^tU(t,s)c(s)g(y(s))ds \\
& +\int_0^tU(t,s)(\partial_tg)(y(s))ds+U(t,0)g(x) .
\end{align*}
So
\begin{align*}
\partial_tF(y)(t) = & A(t)U(t,0)x-\lambda iU(t,0)g(x) \\
&-\lambda\int_0^t\int_s^tU(t,\tau)[2i(D_j\nabla\widetilde{b}^j(\tau)+\nabla\widetilde{b}^j(\tau)D_j+\nabla\widetilde{c}(\tau))\nabla(U(\tau,s)g(y(s))) \\
& \qquad \qquad \qquad \qquad  +i(D_j\Delta\widetilde{b}^j(\tau)+\Delta\widetilde{b}^j(\tau)D_j+\Delta\widetilde{c}(\tau))U(\tau,s)g(y(s))]d\tau ds \\
& -\lambda \int_0^tb(t)\cdot\nabla U(t,s)g(y(s))ds-\lambda \int_0^tc(t)U(t,s)g(y(s))ds \\
&+\lambda \int_0^tU(t,s)b(s)\cdot \nabla g(y(s))ds+\lambda \int_0^tU(t,s)c(s)g(y(s))ds \\
& -\lambda i\int_0^tU(t,s)(\partial_tg)(y(s))ds,
\end{align*}
where
\begin{align*}
&A(t)U(t,0)x \\
&=-i(\Delta +b(t,\xi) \cdot \nabla+c(t,\xi))U(t,0)x \\
&=-i\Delta U(t,0)x-ib(t,\xi)\cdot \nabla U(t,0)x-ic(t,\xi)U(t,0)x \\
&=-i\Delta U(t,0)x-2i(\nabla W)(t,\xi)\cdot \nabla U(t,0)x-i\sum_{j=1}^d(\partial_jW(t,\xi))^2U(t,0)x-i(\Delta W)(t,\xi)U(t,0)x. 
\end{align*}
Therefore,
\begin{align}
\label{testi}
\|\partial_tF(y)\|_{str(\tau_1)}&\lesssim \|\Delta U(\cdot,0)x\|_{L^{\infty}(0,\tau_1;L^2)\cap L^{q}(0,\tau_1;L^{p})}+\|(\nabla W)\cdot \nabla U(\cdot,0) x\|_{L^{\infty}(0,\tau_1;L^2)\cap L^{q}(0,\tau_1;L^{p})} \nonumber \\
& \quad +\left\|\sum_{j=1}^d(\partial_jW)^2U(\cdot,0)x\right\|_{L^{\infty}(0,\tau_1;L^2)\cap L^{q}(0,\tau_1;L^{p})} \nonumber \\
& \quad +\|(\Delta W)U(\cdot,0)x\|_{L^{\infty}(0,\tau_1;L^2)\cap L^{q}(0,\tau_1;L^{p})}+\|g(x)\|_{L^2} \nonumber \\
& \quad +\left\|\int_0^t\int_s^tU(t,\tau)[2i(D_j\nabla\widetilde{b}^j(\tau)+\nabla\widetilde{b}^j(\tau)D_j+\nabla\widetilde{c}(\tau))\nabla(U(\tau,s)g(y(s)))\right. \nonumber \\
& \qquad \qquad \left.+i(D_j\Delta\widetilde{b}^j(\tau)+\Delta\widetilde{b}^j(\tau)D_j+\Delta\widetilde{c}(\tau))U(\tau,s)g(y(s))]d\tau ds\right\|_{L^{\infty}(0,\tau_1;L^2)\cap L^{q}(0,\tau_1;L^{p})} \nonumber \\
& \quad +\left\|\int_0^tb(t)\cdot\nabla U(t,s)g(y)ds\right\|_{L^{\infty}(0,\tau_1;L^2)\cap L^{q}(0,\tau_1;L^{p})} \nonumber \\
& \quad +\left\|\int_0^tc(t)U(t,s)g(y)ds\right\|_{L^{\infty}(0,\tau_1;L^2)\cap L^{q}(0,\tau_1;L^{p})} \nonumber \\
& \quad +\left\|\int_0^tU(t,s)b(s)\cdot \nabla g(y)ds\right\|_{L^{\infty}(0,\tau_1;L^2)\cap L^{q}(0,\tau_1;L^{p})} \nonumber \\
& \quad +\left\|\int_0^tU(t,s)c(s)g(y)ds\right\|_{L^{\infty}(0,\tau_1;L^2)\cap L^{q}(0,\tau_1;L^{p})}+\|\partial_tg(y)\|_{L^{q'}(0,\tau_1;L^{p'})}.
\end{align}
We estimate each term of the right-hand side as follows
\begin{align}
\label{gx}
\|\Delta U(\cdot,0)x\|&_{L^{\infty}(0,\tau_1;L^2)\cap L^{q}(0,\tau_1;L^{p})}\lesssim \|x\|_{H^2}, \nonumber \\
\|(\nabla W)\cdot \nabla U(\cdot,0) x\|_{L^{\infty}(0,\tau_1;L^2)\cap L^{q}(0,\tau_1;L^{p})}&\lesssim \|\nabla W\|_{L^{\infty}(0,\tau_1;L^{\infty})}\|\nabla U(\cdot,0) x\|_{L^{\infty}(0,\tau_1;L^2)\cap L^{q}(0,\tau_1;L^{p})} \nonumber \\
&\lesssim \|x\|_{H^2}, \nonumber \\
\left\|\sum_{j=1}^d(\partial_jW)^2U(\cdot,0)x\right\|&_{L^{\infty}(0,\tau_1;L^2)\cap L^{q}(0,\tau_1;L^{p})}\lesssim \|x\|_{H^2}, \nonumber \\
\|(\Delta W)U(\cdot,0)x\|&_{L^{\infty}(0,\tau_1;L^2)\cap L^{q}(0,\tau_1;L^{p})}\lesssim \|x\|_{H^2}, \nonumber \\
\|g(x)\|_{L^2}&=\|x\|^{\alpha}_{L^{2\alpha}}\lesssim \|x\|^{\alpha}_{H^2}, \\
\left\|\int_0^tb(t)\cdot\nabla U(t,s)g(y)ds\right\|_{L^{\infty}(0,\tau_1;L^2)\cap L^{q}(0,\tau_1;L^{p})}&\lesssim \|b\|_{L^{\infty}(0,\tau_1;L^{\infty})}\left\|\nabla\int_0^t U(t,s)g(y)ds\right\|_{L^{\infty}(0,\tau_1;L^2)\cap L^{q}(0,\tau_1;L^{p})} \nonumber \\
&\lesssim \|g(y)\|_{L^{q'}(0,\tau_1;L^{p'})}+\|\nabla g(y)\|_{L^{q'}(0,\tau_1;L^{p'})} \lesssim \tau^{\theta}_1M_1^{\alpha}, \nonumber \\
\left\|\int_0^tc(t)U(t,s)g(y)ds\right\|_{L^{\infty}(0,\tau_1;L^2)\cap L^{q}(0,\tau_1;L^{p})}&\lesssim \|c\|_{L^{\infty}(0,\tau_1;L^{\infty})}\|g(y)\|_{L^{q'}(0,\tau_1;L^{p'})}\lesssim \tau^{\theta}_1M_1^{\alpha}, \nonumber \\
\left\|\int_0^tU(t,s)b(s)\cdot \nabla g(y)ds\right\|_{L^{\infty}(0,\tau_1;L^2)\cap L^{q}(0,\tau_1;L^{p})}&\lesssim \|b\|_{L^{\infty}(0,\tau_1;L^{\infty})}\left\|\int_0^tU(t,s)\nabla g(y)ds\right\|_{L^{\infty}(0,\tau_1;L^2)\cap L^{q}(0,\tau_1;L^{p})} \nonumber \\
&\lesssim \|g(y)\|_{L^{q'}(0,\tau_1;L^{p'})}+\|\nabla g(y)\|_{L^{q'}(0,\tau_1;L^{p'})}\lesssim \tau^{\theta}_1M_1^{\alpha}, \nonumber \\
\left\|\int_0^tU(t,s)c(s)g(y)ds\right\|_{L^{\infty}(0,\tau_1;L^2)\cap L^{q}(0,\tau_1;L^{p})}&\lesssim \|c\|_{L^{\infty}(0,\tau_1;L^{\infty})}\|g(y)\|_{L^{q'}(0,\tau_1;L^{p'})}\lesssim \tau^{\theta}_1M_1^{\alpha}, \nonumber \\
\|\partial_tg(y)\|_{L^{q'}(0,\tau_1;L^{p'})} \lesssim & \tau_1^{\theta}\|y\|^{\alpha-1}_{L^{\infty}(0,\tau_1;L^{2\alpha})}\|\partial_ty\|_{L^{q}(0,\tau_1;L^{p})} \lesssim \tau_1^{\theta}M_1^{\alpha}. \nonumber 
\end{align}
Here we used H$\ddot{\text{o}}$lder's inequality with
\[ 1-\frac{1}{p}=\frac{\alpha-1}{2\alpha}+\frac{1}{p}, \quad 1-\frac{1}{q}=\theta +\frac{1}{q}, \]
and  $q=\frac{8\alpha}{d(\alpha-1)}>2$ from the assumption of $\alpha$, and so $\theta>0$. We consider the 6th term on the right-hand side of (\ref{testi}).
From the Strichatrz estimate and $\phi_j$ is a first order polynomial, we have
\begin{align*}
& \left\|\int_0^t\int_s^tU(t,\tau)[2i(D_j\nabla\widetilde{b}^j(\tau)+\nabla\widetilde{b}^j(\tau)D_j+\nabla\widetilde{c}(\tau))\nabla (U(\tau,s)g(y(s))) \right. \\
& \qquad \qquad \left. +i(D_j\Delta\widetilde{b}^j(\tau)+\Delta\widetilde{b}^j(\tau)D_j+\Delta\widetilde{c}(\tau))U(\tau,s)g(y(s))]d\tau ds\right\|_{L^{\infty}(0,\tau_1;L^2)\cap L^{q}(0,\tau_1;L^{p})} \\
&\lesssim \left|\int_0^{\tau_1}\left\|\int_s^tU(t,\tau)[2i(D_j\nabla\widetilde{b}^j(\tau)+\nabla\widetilde{c}(\tau))\nabla (U(\tau,s)g(y(s))) \right.\right. \\
& \qquad \qquad \left.\left. +i(D_j\Delta\widetilde{b}^j(\tau)+\Delta\widetilde{c}(\tau))U(\tau,s)g(y(s))]d\tau\right\|_{L^{\infty}(0,\tau_1;L^2)\cap L^{q}(0,\tau_1;L^{p})} ds\right| \\
&\lesssim \left|\int_0^{\tau_1}\left\|2i(D_j\nabla\widetilde{b}^j(t)+\nabla\widetilde{c}(t))\nabla (U(t,s)g(y(s))) \right.\right. \\
& \qquad \qquad \left.\left. +i(D_j\Delta\widetilde{b}^j(t)+\Delta\widetilde{c}(t))U(t,s)g(y(s))\right\|_{L^{\infty}(0,\tau_1;L^2)} ds\right| \\
&\lesssim \|D_j\nabla\widetilde{b}^j(t)+\nabla\widetilde{c}(t)\|_{L^{\infty}(0,\tau_1;L^{\infty})}\left|\int_0^{\tau_1}\|\nabla (U(t,s)g(y(s)))\|_{L^{\infty}(0,\tau_1;L^2)} ds \right| \\
& \qquad \qquad +\|D_j\Delta\widetilde{b}^j(t)+\Delta\widetilde{c}(t)\|_{L^{\infty}(0,\tau_1;L^{\infty})}\left|\int_0^{\tau_1}\|U(t,s)g(y(s))\|_{L^{\infty}(0,\tau_1;L^2)}ds\right| \\
&\lesssim \left|\int_0^{\tau_1}\|U(t,s)g(y(s))\|_{L^{\infty}(0,\tau_1;H^1)}ds \right|\lesssim\tau_1\|g(y)\|_{L^{\infty}(0,\tau_1;H^1)}\lesssim\tau_1M_1^{\alpha}.
\end{align*}
Since we see
\begin{align}
\label{yLip}
\|y(t)-y(s)\|_{L^2}&=\left\| \int_s^t(\partial_ty)(\tau)d\tau \right\|_{L^2} \le \left| \int_s^t\|(\partial_ty)(\tau)\|_{L^2}d\tau \right| \nonumber \\
&\le |t-s|\|\partial_ty\|_{L^{\infty}(0,\tau_1;L^2)} \le M_1|t-s|,
\end{align}
$y$ is Lipschitz continuous. Also, by the interpolation inequality with
\[ \frac{1}{2\alpha}=\frac{\theta}{2}+(1-\theta)\left( \frac{1}{2}-\frac{2}{d}\right), \]
we have
\begin{align}
\label{yHol}
\|y(t)-y(s)\|_{L^{2\alpha}} &\le \|y(t)-y(s)\|^{\theta}_{L^2}\|y(t)-y(s)\|^{1-\theta}_{H^2} \nonumber \\
&\lesssim M_1^{\theta}|t-s|^{\theta}M_1^{1-\theta}=M_1|t-s|^{\theta}.
\end{align}
Similarly, by the interpolation inequality with
\[ \frac{1}{2\alpha}=\frac{\theta}{2}+(1-\theta)\left( \frac{1}{2}-\frac{1}{d}\right), \]
we have
\begin{align*}
\|\nabla(y(t)-y(s))\|_{L^{2\alpha}} &\le \|\nabla(y(t)-y(s))\|^{\theta}_{L^2}\|\nabla(y(t)-y(s))\|^{1-\theta}_{H^1} \nonumber \\
&\lesssim \|y(t)-y(s)\|^{\theta'}_{L^2}\|y(t)-y(s)\|^{1-\theta'}_{H^2} \nonumber \\
&\lesssim M_1^{\theta'}|t-s|^{\theta'}M_1^{1-\theta'}=M_1|t-s|^{\theta'}, \quad \theta'\le\theta,
\end{align*}
where we use Sobolev's embedding in $1<\alpha\le 1+\frac{2}{(d-2)^+}$.
So, $y$ is H$\ddot{\text{o}}$lder continuous. Therefore,
\begin{align*}
\|g(x)-g(y)\|_{L^{\infty}(0,\tau_1;H^1)} &\lesssim (\|x\|^{\alpha-1}_{L^{2\alpha}}+\|y\|^{\alpha-1}_{L^{\infty}(0,\tau_1;L^{2\alpha})})\|x-y\|_{L^{\infty}(0,\tau_1;W^{1,2\alpha})} \\
&\lesssim (1+\|x\|^{\alpha-1}_{H^2}+M_1^{\alpha-1})M_1\tau_1^{\theta}.
\end{align*}
In the end, the following holds,
\[ \|\partial_tF(y)\|_{str(\tau_1)} \lesssim \|x\|_{H^2}+\|x\|^{\alpha}_{H^2}+\tau_1^{\theta}M_1^{\alpha}+\tau_1M_1^{\alpha}+(1+\|x\|^{\alpha-1}_{H^2}+M_1^{\alpha-1})M_1\tau_1^{\theta}. \]
Similarly,
\begin{align*}
\|F(y)\|_{str(\tau_1)} \lesssim & \|x\|_{H^2}+\|g(y)\|_{L^{q'}(0,\tau_1;L^{p'})}\lesssim \|x\|_{H^2}+\tau_1^{\theta}M_1^{\alpha}.
\end{align*}
Therefore,
\[ \|F(y)\|_{str(\tau_1)}+\|\partial_tF(y)\|_{str(\tau_1)} \lesssim \|x\|_{H^2}+\|x\|^{\alpha}_{H^2}+\tau_1^{\theta}M_1^{\alpha}+\tau_1M_1^{\alpha}+(1+\|x\|^{\alpha-1}_{H^2}+M_1^{\alpha-1})M_1\tau_1^{\theta}. \]
Next, we estimate $\|\Delta F(y)\|_{L^{\infty}(0,\tau_1;L^2)}$. 
From (\ref{partialt1}) and (\ref{partialt2}), we have
\begin{align*}
&\int_0^tA(t)U(t,s)g(y(s))ds+g(y(t)) \\ 
&=\int_0^t\{A(t)U(t,s)-U(t,s)A(s)\}g(y(s))ds+\int_0^tU(t,s)(\partial_tg)(y(s))ds+U(t,0)g(x),
\end{align*}
then
\begin{eqnarray*}
\int_0^tU(t,s)A(s)g(y(s))ds=U(t,0)g(y(0))-g(y(t))+\int_0^tU(t,s)(\partial_tg)(y(s))ds.
\end{eqnarray*}
By
\begin{align*}
U(t,s)A(s)g(y(s))&=-iU(t,s)(\Delta +b(s,\xi)\cdot \nabla +c(s,\xi))g(y(s)) \\
&=-iU(t,s)(b(s,\xi)\cdot \nabla +c(s,\xi))g(y(s))-i\Delta U(t,s)g(y(s)) \\
&\quad +i\int_s^tU(t,\tau)[2i(D_j\nabla\widetilde{b}^j(\tau)+\nabla\widetilde{b}^j(\tau)D_j+\nabla\widetilde{c}(\tau))\nabla (U(\tau,s)g(y(s))) \\
& \qquad \qquad \qquad \qquad  +i(D_j\Delta\widetilde{b}^j(\tau)+\Delta\widetilde{b}^j(\tau)D_j+\Delta\widetilde{c}(\tau))U(\tau,s)g(y(s))]d\tau,
\end{align*}
we have 
\begin{align*}
\Delta \int_0^tU(t,s)g(y(s))ds&=iU(t,0)g(y(0))-ig(y(t))+i\int_0^tU(t,s)(\partial_tg)(y(s))ds \\
& \quad -\int_0^tU(t,s)b(s,\xi)\cdot \nabla g(y(s))ds-\int_0^tU(t,s)c(s,\xi)g(y(s))ds \\
& \quad -\int_0^t\int_s^tU(t,\tau)[2i(D_j\nabla\widetilde{b}^j(\tau)+\nabla\widetilde{b}^j(\tau)D_j+\nabla\widetilde{c}(\tau))\nabla (U(\tau,s)g(y(s))) \\
& \qquad \qquad \qquad \qquad  +i(D_j\Delta\widetilde{b}^j(\tau)+\Delta\widetilde{b}^j(\tau)D_j+\Delta\widetilde{c}(\tau))U(\tau,s)g(y(s))]d\tau ds.
\end{align*}
Therefore,
\begin{align}
\label{DF}
& \|\Delta F(y)\|_{L^{\infty}(0,\tau_1;L^2)} \nonumber \\
&\lesssim \|x\|_{H^2}+\|U(t,0)g(y(0))-g(y(t))\|_{L^{\infty}(0,\tau_1;L^2)}+\left\|\int_0^tU(t,s)(\partial_tg)(y)ds\right\|_{L^{\infty}(0,\tau_1;L^2)} \nonumber \\ 
& \quad +\left\| \int_0^tU(t,s)b(s,\xi)\cdot \nabla g(y)ds\right\|_{L^{\infty}(0,\tau_1;L^2)}+\left\| \int_0^tU(t,s)c(s,\xi)g(y)ds\right\|_{L^{\infty}(0,\tau_1;L^2)} \nonumber \\
& \quad +\left\| \int_0^t\int_s^tU(t,\tau)[2i(D_j\nabla\widetilde{b}^j(\tau)+\nabla\widetilde{b}^j(\tau)D_j+\nabla\widetilde{c}(\tau))\nabla (U(\tau,s)g(y(s))) \right. \nonumber \\
& \qquad \qquad \qquad \qquad \left. +i(D_j\Delta\widetilde{b}^j(\tau)+\Delta\widetilde{b}^j(\tau)D_j+\Delta\widetilde{c}(\tau))U(\tau,s)g(y(s))]d\tau ds \right\|_{L^{\infty}(0,\tau_1;L^2)} \nonumber \\
&\lesssim \|x\|_{H^2}+\|U(t,0)g(y(0))-g(y(t))\|_{L^{\infty}(0,\tau_1;L^2)}+\left\|\int_0^tU(t,s)(\partial_tg)(y)ds\right\|_{L^{\infty}(0,\tau_1;L^2)} \nonumber \\
& \quad +\|b\|_{L^{\infty}(0,\tau_1;L^{\infty})}\|\nabla g(y)\|_{L^{q'}(0,\tau_1;L^{p'})}+\|c\|_{L^{\infty}(0,\tau_1;L^{\infty})}\|g(y)\|_{L^{q'}(0,\tau_1;L^{p'})} \nonumber \\
& \quad +\left| \int_0^t\left\| \int_s^tU(t,\tau)[2i(D_j\nabla\widetilde{b}^j(\tau)+\nabla\widetilde{b}^j(\tau)D_j+\nabla\widetilde{c}(\tau))\nabla (U(\tau,s)g(y(s))) \right.\right. \nonumber \\
& \qquad \qquad \qquad \qquad \left.\left. +i(D_j\Delta\widetilde{b}^j(\tau)+\Delta\widetilde{b}^j(\tau)D_j+\Delta\widetilde{c}(\tau))U(\tau,s)g(y(s))]d\tau\right\|_{L^{\infty}(0,\tau_1;L^2)} ds \right| \nonumber \\
&\lesssim \|x\|_{H^2}+\|U(t,0)g(y(0))-g(y(t))\|_{L^{\infty}(0,\tau_1;L^2)}+\left\|\int_0^tU(t,s)(\partial_tg)(y)ds\right\|_{L^{\infty}(0,\tau_1;L^2)} \nonumber \\
& \quad +\|\nabla g(y)\|_{L^{q'}(0,\tau_1;L^{p'})}+\|g(y)\|_{L^{q'}(0,\tau_1;L^{p'})}+\tau_1\|g(y)\|_{L^{\infty}(0,\tau_1;H^1)}.
\end{align}
Therefore, we have the following
\begin{align*}
\left\|\int_0^tU(t,s)\partial_tg(y)ds\right\|_{L^{\infty}(0,\tau_1;L^2)} \lesssim \|\partial_tg(y)\|_{L^{q'}(0,\tau_1;L^{p'})}\lesssim \tau^{\theta}_1M_1^{\alpha}.
\end{align*}
Next, we estimate the second terms of the last line of (\ref{DF}).
\begin{align}
\label{2kou}
& \|U(t,0)g(y(0))-g(y(t))\|_{L^{\infty}(0,\tau_1;L^2)} \nonumber \\
& \le \|U(t,0)g(x)-g(x)\|_{L^{\infty}(0,\tau_1;L^2)}+\|g(x)-g(y(t))\|_{L^{\infty}(0,\tau_1;L^2)}.
\end{align}
First, we estimate the first terms in the last line of (\ref{2kou}). From (\ref{gx}), we have \\ $g(x)\in L^{\infty}(0,\tau_1;L^2)$. We estimate
\begin{align*}
\|U(t,0)g(x)-g(x)\|_{L^{\infty}(0,\tau_1;L^2)} &=\|(U(t,0)-I)g(x)\|_{L^{\infty}(0,\tau_1;L^2)} \\
& \le \|U(t,0)-I\|\|g(x)\|_{L^2} \\
& \lesssim \|g(x)\|_{L^2} \lesssim \|x\|_{H^2}^{\alpha}.
\end{align*}
Next, we estimate the second terms in the last line of (\ref{2kou}). By (\ref{yLip}) and (\ref{yHol}), we have
\begin{align*}
\|g(x)-g(y(t))\|_{L^{\infty}(0,\tau_1;L^2)} &\lesssim (\|x\|^{\alpha-1}_{L^{2\alpha}}+\|y\|^{\alpha-1}_{L^{\infty}(0,\tau_1;L^{2\alpha})})\|x-y\|_{L^{\infty}(0,\tau_1;L^{2\alpha})} \\
&\lesssim (1+\|x\|^{\alpha-1}_{H^2}+M_1^{\alpha-1})M_1\tau_1^{\theta}.
\end{align*}
So,
\[ \|\Delta F(y)\|_{L^{\infty}(0,\tau_1;L^2)} \lesssim \|x\|_{H^2}+\|x\|_{H^2}^{\alpha}+\tau_1^{\theta}M_1^{\alpha}+\tau_1M_1^{\alpha}+(1+\|x\|^{\alpha-1}_{H^2}+M_1^{\alpha-1})M_1\tau_1^{\theta}. \]
Therefore,
\begin{align*}
& \|F(y)\|_{str(\tau_1)}+\|\partial_tF(y)\|_{str(\tau_1)}+\|\Delta F(y)\|_{L^{\infty}(0,\tau_1;L^2)} \\
& \le C_{\tau_1}(\|x\|_{H^2}+\|x\|^{\alpha}_{H^2}+\tau_1^{\theta}M_1^{\alpha}+\tau_1M_1^{\alpha}+(1+\|x\|^{\alpha-1}_{H^2}+M_1^{\alpha-1})M_1\tau_1^{\theta}),
\end{align*}
where $C_{\tau_1}$ is a constant that depends on $\tau_1$. 

We shall choose $M_1$ and $\tau_1$ to obtain $F(\mathcal{Y}^{\tau_1}_{M_1})\subset \mathcal{Y}^{\tau_1}_{M_1}$ such that
\[ C_{\tau_1}(\|x\|_{H^2}+\|x\|^{\alpha}_{H^2}+\tau_1^{\theta}M_1^{\alpha}+\tau_1M_1^{\alpha}+(1+\|x\|^{\alpha-1}_{H^2}+M_1^{\alpha-1})M_1\tau_1^{\theta}) \le M_1. \]
To this end, we define the real-valued continuous, $(\mathcal{F}_t)$-adapted process 
\[ Z_t^{(1)}=2^{\alpha-1}C_t^{\alpha}(\|x\|_{H^2}+\|x\|_{H^2}^{\alpha})^{\alpha-1}(2t^{\theta}+t)+C_tt^{\theta}(1+\|x\|^{\alpha-1}_{H^2}), \ t\in [0,T], \]
choose the $(\mathcal{F}_t)$-stopping time as
\[ \tau_1=\inf \left\{ t\in[0,T] ;Z_t^{(1)}>\frac{1}{2} \right\} \wedge T, \]
and set $M_1=2C_{\tau_1}(\|x\|_{H^2}+\|x\|_{H^2}^{\alpha})$. Then it follows that $Z_{\tau_1}^{(1)}\le \frac{1}{2}$ and $F(\mathcal{Y}^{\tau_1}_{M_1})\subset \mathcal{Y}^{\tau_1}_{M_1}$. 

Moreover, since $|g(y_1)-g(y_2)|\le \alpha(|y_1|^{\alpha-1}+|y_2|^{\alpha-1})|y_1-y_2|$, for $y_1,y_2\in \mathcal{Y}^{\tau_1}_{M_1}$
\begin{align}
\label{gesti}
\|F(y_1)-F(y_2)\|_{str(\tau_1)} &\le C_{\tau_1}\|g(y_1)-g(y_2)\|_{L^{q'}(0,\tau_1;L^{p'})} \nonumber \\
&\le C_{\tau_1}(\|y_1\|^{\alpha-1}_{L^{\infty}(0,\tau_1;L^{2\alpha})}+\|y_2\|^{\alpha-1}_{L^{\infty}(0,\tau_1;L^{2\alpha})})\|y_1-y_2\|_{L^{q'}(0,\tau_1;L^p)} \\
&\le C_{\tau_1}\tau_1^{\theta}M_1^{\alpha-1}\|y_1-y_2\|_{L^q(0,\tau_1;L^p)} \nonumber \\
&\le \frac{1}{2}\|y_1-y_2\|_{L^q(0,\tau_1;L^p)}, \nonumber 
\end{align}
which implies that $F$ is a contraction in $L^{q}(0,\tau_1;L^p)$. 
Since $\mathcal{Y}^{\tau_1}_{M_1}$ is a complete metric subspace in $L^{\infty}(0,\tau_1;L^2)$, Banach's fixed point theorem yields a unique $y\in \mathcal{Y}^{\tau_1}_{M_1}$ with $y=F(y)$ on $[0,\tau_1]$. 
Consequently, setting $y_1(t):=y(t\wedge \tau_1), \ t\in [0,T]$, we deduce that $(y_1,\tau_1)$ is a local solution of (\ref{RSNLS}), such that $y_1(t)=y_1(t\wedge \tau_1), \ t\in [0,T],$ and $y_1|_{[0.\tau_1]}\in C([0,\tau_1];H^2).$

\textbf{Step 2.} We use an induction argument. Suppose that at the $n$th step we have a local solution $(y_n,\tau_n)$ of (\ref{RSNLS}), such that $\tau_n\ge \tau_{n-1}, \ y_n(t)=y_n(t\wedge \tau_n), \ t\in[0,T],$ and $y_n|_{[0,\tau_n]}\in C([0,\tau_n];H^2).$ We construct $(y_{n+1},\tau_{n+1})$. Set
\begin{eqnarray*}
\mathcal{Y}^{\sigma_n}_{M_{n+1}}=\{ z\in L^{\infty}(0,\sigma_n;L^2);\|z\|_{str(\sigma_n)}+\|\partial_tz\|_{str(\sigma_n)}+\|\Delta z\|_{L^{\infty}(0,\sigma_n;L^2)}\le M_{n+1} \},
\end{eqnarray*}
and define the map $F_n$ by
\begin{align}
F_n(z)(t)=U(\tau_n+t,\tau_n)y_n(\tau_n)-\lambda i\int_0^tU(\tau_n+t,\tau_n+s)g(z(s))ds.
\end{align}
Analogous calculations as in Step 1 show that for $z\in \mathcal{Y}^{\sigma_n}_{M_{n+1}}$
\begin{align*}
& \|F_n(z)\|_{str(\sigma_n)}+\|\partial_tF_n(z)\|_{str(\sigma_n)}+\|\Delta F_n(z)\|_{L^{\infty}(0,\sigma_n;L^2)} \\
&\le C_{\tau_n+\sigma_n}(\|y_n(\tau_n)\|_{H^2}+\|y_n(\tau_n)\|^{\alpha}_{H^2}+\sigma_n^{\theta}M_{n+1}^{\alpha}+\sigma_nM_{n+1}^{\alpha}+(1+\|y_n(\tau_n)\|^{\alpha-1}_{H^2}+M_{n+1}^{\alpha-1})M_{n+1}\sigma_n^{\theta}).
\end{align*}
We shall choose $M_{n+1}$ and $\sigma_n$ to obtain $F_n(\mathcal{Y}_{M_{n+1}}^{\sigma_n})\subset \mathcal{Y}_{M_{n+1}}^{\sigma_n}$ such that
\[ C_{\tau_n+\sigma_n}(\|y_n(\tau_n)\|_{H^2}+\|y_n(\tau_n)\|^{\alpha}_{H^2}+\sigma_n^{\theta}M_{n+1}^{\alpha}+\sigma_nM_{n+1}^{\alpha}+(1+\|y_n(\tau_n)\|^{\alpha-1}_{H^2}+M_{n+1}^{\alpha-1})M_{n+1}\sigma_n^{\theta}) \le M_{n+1}. \]
To this end, we define the real-valued continuous, $(\mathcal{F}_{\tau_n+t})$-adapted process 
\[ Z_t^{(n+1)}=2^{\alpha-1}C_{\tau_n+t}^{\alpha}(\|y_n(\tau_n)\|_{H^2}+\|y_n(\tau_n)\|_{H^2}^{\alpha})^{\alpha-1}(2t^{\theta}+t)+C_{\tau_n+t}t^{\theta}(1+\|y_n(\tau_n)\|^{\alpha-1}_{H^2}), \ t\in [0,T], \]
choose the $(\mathcal{F}_{\tau_n+t})$-stopping time as
\[ \sigma_n=\inf \left\{ t\in[0,T-\tau_n] ;Z_t^{(n+1)}>\frac{1}{2} \right\} \wedge (T-\tau_n), \]
and set $M_{n+1}=2C_{\tau_n+\sigma_n}(\|y_n(\tau_n)\|_{H^2}+\|y_n(\tau_n)\|_{H^2}^{\alpha})$. Then it follows that $Z_{\sigma_n}^{(n+1)}\le \frac{1}{2}$ and $F_n(\mathcal{Y}^{\sigma_n}_{M_{n+1}})\subset \mathcal{Y}^{\sigma_n}_{M_{n+1}}$. 

Moreover, since $|g(z_1)-g(z_2)|\le \alpha(|z_1|^{\alpha-1}+|z_2|^{\alpha-1})|z_1-z_2|$, for $z_1,z_2 \in \mathcal{Y}^{\sigma_n}_{M_{n+1}}$ 
\begin{align*}
\|F_n(z_1)-F_n(z_2)\|_{str(\sigma_n)} &\le C_{\tau_n+\sigma_n}\|g(z_1)-g(z_2)\|_{L^{q'}(0,\sigma_n;L^{p'})} \\
& \le C_{\tau_n+\sigma_n}(\|z_1\|^{\alpha-1}_{L^{\infty}(0,\sigma_n;L^{2\alpha})}+\|z_2\|^{\alpha-1}_{L^{\infty}(0,\sigma_n;L^{2\alpha})})\|z_1-z_2\|_{L^{q'}(0,\sigma_n;L^p)} \\
& \le C_{\tau_n+\sigma_n}\sigma_n^{\theta}M_{n+1}^{\alpha-1}\|z_1-z_2\|_{L^q(0,\sigma_n;L^p)} \\
& \le \frac{1}{2}\|z_1-z_2\|_{L^q(0,\sigma_n;L^p)},
\end{align*}
which implies that $F_n$ is a contraction from $L^{\infty}(0,\sigma_n;L^2)$ to the same space. 

Set $\tau_{n+1}=\tau_n+\sigma_n$. Then, similarly to the proof of Lemma 4.2 in \cite{BRZ14}, we can show $\tau_{n+1}$ is an $\Fi$-stopping time. 
By Banach's fixed point theorem, there exists a unique $z_{n+1}\in \mathcal{Y}^{\sigma_n}_{M_{n+1}}$ satisfying $z_{n+1}=F_n(z_{n+1})$ on $[0,\sigma_n]$. We define 
\begin{eqnarray*}
y_{n+1}(t)=
\begin{cases}
y_n(t), & t\in [0,\tau_n]; \\
z_{n+1}((t-\tau_n)\wedge \sigma_n), & t\in (\tau_n,T].
\end{cases}
\end{eqnarray*}
It follows from the definition of $F$ and $F_n$ that $y_{n+1}=F(y_{n+1})$ on $[0,\tau_{n+1}]$. 

And, similar to the proof of Lemma 6.2 in \cite{BRZ14}, $y_{n+1}$ is an adapted to $\Fi$ in $H^2$. 
Hence, $(y_{n+1},\tau_{n+1})$ is a local solution of (\ref{RSNLS}), such that $y_{n+1}(t)=y_{n+1}(t\wedge \tau_{n+1}), \ t\in[0,T]$, and $y_{n+1}|_{[0,\tau_{n+1}]}\in C([0,\tau_{n+1}];H^2)$. 
Starting from Step 1 and repeating the procedure in Step 2, we finally construct a sequence of local solutions $(y_n,\tau_n), \ n\in \mathbb{N},$ where $\tau_n$ are increasing stopping times and $y_{n+1}=y_n, \ \text{on} \ [0,\tau_n]$. 

To prove the uniqueness, for any two local solutions $(\widetilde{y}_i,\sigma_i), \ i=1,2,$ define $\iota=\sup \{ t\in[0,\sigma_1\wedge \sigma_2]:\widetilde{y}_1=\widetilde{y}_2 \ \text{on} \ [0,t] \}$. Suppose that $\Pro (\iota <\sigma_1\wedge \sigma_2)>0$. For $\omega \in \{ \iota<\sigma_1\wedge \sigma_2 \},$ we have $\widetilde{y}_1(\omega)=\widetilde{y}_2(\omega)$ on $[0,\iota(\omega)]$ by the continuity in $H^2$, and for $t\in [0,\sigma_1\wedge \sigma_2(\omega)-\iota(\omega))$ 
\begin{align*}
& \|\widetilde{y}_1(\omega)-\widetilde{y}_2(\omega)\|_{L^{\infty}(\iota(\omega),\iota(\omega)+t;L^2)}+\|\widetilde{y}_1(\omega)-\widetilde{y}_2(\omega)\|_{L^q(\iota(\omega),\iota(\omega)+t;L^p)} \\
& =\|F(\widetilde{y}_1(\omega))-F(\widetilde{y}_2(\omega))\|_{L^{\infty}(\iota(\omega),\iota(\omega)+t;L^2)}+\|F(\widetilde{y}_1(\omega))-F(\widetilde{y}_2(\omega))\|_{L^q(\iota(\omega),\iota(\omega)+t;L^p)} \\
& \le C_{\iota(\omega)+t}\|g(\widetilde{y}_1(\omega))-g(\widetilde{y}_2(\omega))\|_{L^{q'}(\iota(\omega),\iota(\omega)+t;L^{p'})} \\
& \le C_{\iota(\omega)+t}(\|\widetilde{y}_1(\omega)\|^{\alpha-1}_{L^{\beta}(\iota(\omega),\iota(\omega)+t;L^{2\alpha})}+\|\widetilde{y}_2(\omega)\|^{\alpha-1}_{L^{\beta}(\iota(\omega),\iota(\omega)+t;L^{2\alpha})})\|\widetilde{y}_1(\omega)-\widetilde{y}_2(\omega)\|_{L^{q'+\varepsilon}(\iota(\omega),\iota(\omega)+t;L^p)} \\
& \le C_{\iota(\omega)+t}\widetilde{M}(t)t^{\theta}\|\widetilde{y}_1(\omega)-\widetilde{y}_2(\omega)\|_{L^q(\iota(\omega),\iota(\omega)+t;L^p)},
\end{align*}
where $\varepsilon>0$, $\beta=\frac{q'(q'+\varepsilon)(\alpha-1)}{\varepsilon}<\infty$, $\widetilde{M}(t):=\|\widetilde{y}_1(\omega)\|^{\alpha-1}_{L^{\beta}(\iota(\omega),\iota(\omega)+t;H^2)}+\|\widetilde{y}_2(\omega)\|^{\alpha-1}_{L^{\beta}(\iota(\omega),\iota(\omega)+t;H^2)} \to 0$ as $t\to 0$. Therefore, with $t$ small enough we deduce that $\widetilde{y}_1(\omega)=\widetilde{y}_2(\omega)$ on $[\iota(\omega),\iota(\omega)+t]$, hence $\widetilde{y}_1(\omega)=\widetilde{y}_2(\omega)$ on $[0,\iota(\omega)+t]$, which contradicts the definition of $\iota$.

Finally, we prove the continuous dependence of the initial data. Suppose that $x_m\to x$ in $H^2$ and let $(y_m,(\tau_n^m)_{n\in \mathbb{N}},\tau^*(x_m))$ be the unique local solutions of (\ref{RSNLS}) corresponding to the initial data $x_m, \ m\ge 1.$ Since $\|x_m\|_{H^2}\le \|x\|_{H^2}+1$ for $m\ge m_1$ with $m_1$ large enough, we modify $\tau_1(\le T)$ in the proof of Proposition \ref{ymain} by 
\begin{align*}
\tau_1=\inf \{ t\in [0,T]:2^{\alpha-1}C_t^{\alpha}((\|x\|_{H^2}+1)&+(\|x\|_{H^2}+1)^{\alpha})^{\alpha-1}(2t^{\theta}+t)  \\
& \left. +C_tt^{\theta}(\|x\|_{H^2}+2)^{\alpha-1}>\frac{1}{2} \right\} \wedge T,
\end{align*}
such that $\tau_1$ is independent for $m\ge m_1$. Hence,
\[ \widetilde{R}:=\sup_{m\ge m_1}\|y_m\|_{L^{\infty}(0,\tau_1;H^2)}<\infty, \quad \Pas. \]

We first prove the continuous dependence on initial data on the interval $[0,\tau_1]$. 
\begin{enumerate}
\renewcommand{\labelenumi}{(\roman{enumi})}
\item Claim 1: $\|y_m-y\|_{L^{\infty}(0,\tau_1;L^2)}\to 0, \ \text{as} \ m\to \infty.$

From
\begin{eqnarray*}
y_m-y=U(t,0)(x_m-x)-\lambda i\int_0^tU(t,s)(g(y_m(s))-g(y(s)))ds,
\end{eqnarray*}
taking $t$ small and independently of $m(\ge m_1)$, we have 
\begin{align*}
\|y_m-y\|_{L^{\infty}(0,t;L^2)} & \lesssim \|x_m-x\|_{L^2}+\|g(y_m)-g(y)\|_{L^{q'}(0,t;L^{p'})} \\
& \lesssim \|x_m-x\|_{L^2}+(\|y_m\|^{\alpha-1}_{L^{\infty}(0,t;L^{2\alpha})}+\|y\|^{\alpha-1}_{L^{\infty}(0,t;L^{2\alpha})})\|y_m-y\|_{L^{q'}(0,t;L^p)} \\
& \lesssim \|x_m-x\|_{L^2}+t^{\theta}\widetilde{R}^{\alpha-1}\|y_m-y\|_{L^q(0,t;L^p)}.
\end{align*}
Therefore,
\[ \|y_m-y\|_{L^{\infty}(0,t;L^2)} \le C_{\tau_1}(\|x_m-x\|_{L^2}+t^{\theta}\widetilde{R}^{\alpha-1}\|y_m-y\|_{L^q(0,t;L^p)}). \]
So, if we choose $t$ satisfy $C_{\tau_1}t^{\theta}\widetilde{R}^{\alpha-1} \le \frac{1}{2}$, we obtain
\[ \|y_m-y\|_{L^{\infty}(0,t;L^2)} \le 2C_{\tau_1}\|x_m-x\|_{L^2}. \]
Since $t$ is independent of $m(\ge m_1)$,
\[ \|y_m-y\|_{L^{\infty}(0,t;L^2)} \to 0, \ \text{as} \ m\to \infty. \]
By repeating this a finite number of times, we can show Claim 1.\\
\item Claim 2: $\|y_m-y\|_{L^{\infty}(0,\tau_1;H^s)}\to 0, \ \text{as} \ m\to \infty, \ 0<s<2.$

From the interpolation inequality,
\begin{align*}
\|y_m-y\|_{L^{\infty}(0,\tau_1;H^s)} & \le \|y_m-y\|_{L^{\infty}(0,\tau_1;L^2)}^{\frac{2-s}{2}}\|y_m-y\|_{L^{\infty}(0,\tau_1;H^2)}^{1-\frac{2-s}{2}} \\
& \le \|y_m-y\|_{L^{\infty}(0,\tau_1;L^2)}^{\frac{2-s}{2}}(\|y_m\|_{L^{\infty}(0,\tau_1;H^2)}+\|y\|_{L^{\infty}(0,\tau_1;H^2)})^{1-\frac{2-s}{2}} \\
& \le \widetilde{R}^{1-\frac{2-s}{2}}\|y_m-y\|_{L^{\infty}(0,\tau_1;L^2)}^{\frac{2-s}{2}}.
\end{align*}
Therefore, if we take $m\to \infty$, we can show Claim 2 from Claim 1. 
\end{enumerate}
Now, since $y_m(\tau_1)\to y(\tau_1) \ \text{in} \ H^2$, similarly we can get the above results on $[\tau_1,\tau_2]$ with $\tau_2$ depending on $\|y(\tau_1)\|_{H^2}$. Therefore, we can show the continuous dependence on $[0,\tau_2]$. Reiterating this procedure, we then obtain increasing stopping times $\tau_n$, depending on $\|y(\tau_{n-1})\|_{H^2}$, such that continuous dependence holds on every $[0,\tau_n]$. Therefore, for $n\ge1, \ \Pas \ \omega \in \Omega$ and $0\le s<2$, the map $x\to y(\cdot,x,\omega)$ is continuous from $H^2$ to $L^{\infty}(0,\tau_n;H^s)$. 
\end{proof}

\begin{proof}[Proof of Theorem \ref{ycor}]
First, for any Strichartz pair $(\rho,\gamma)$, we show $y_n|_{[0,\tau_n]}\in L^{\gamma}(0,\tau_n;W^{2,\rho})$. We consider maps
\[ y(t)=U(t,0)x-\lambda i\int_0^tU(t,s)g(y(s))ds, \]
and 
\begin{align}
\label{Fy2}
F(y)(t)=U(t,0)x-\lambda i\int_0^tU(t,s)g(y(s))ds, \ t\in [0,T].
\end{align}
Let us first consider the case $d=5,6,7$. Choose the Strichartz pair \\ $(p,q)=(\frac{d(\alpha+1)}{d+2\alpha -2},\frac{4(\alpha+1)}{(d-4)(\alpha-1)})$. By Strichartz estimates in Lemma \ref{H2SE},
\begin{align}
\label{eq1}
\|F(y)\|_{L^q(0,T;W^{2,p})}\lesssim \|x\|_{H^2}+\|g(y)\|_{L^{q'}(0,T;W^{2,p'})}. 
\end{align}
Also, we have that
\begin{align}
\label{eq2}
& \|g(y)\|_{L^{q'}(0,T;W^{2,p'})} \nonumber \\
& \lesssim \||y|^{\alpha-1}y\|_{L^{q'}(0,T;L^{p'})}+\||y|^{\alpha-1}|\nabla y|\|_{L^{q'}(0,T;L^{p'})}+\||y|^{\alpha-2}|\nabla y|^2\|_{L^{q'}(0,T;L^{p'})} \nonumber \\
& \quad +\||y|^{\alpha-1}|\Delta y|\|_{L^{q'}(0,T;L^{p'})}.
\end{align}
From H$\ddot{\text{o}}$lder's inequality and the Sobolev embedding it follows that
\begin{align}
\label{eq3}
\||y|^{\alpha-1}y\|_{L^{q'}(0,T;L^{p'})} & \le T^{\theta}\|y\|^{\alpha-1}_{L^q(0,T;L^{\frac{p(\alpha-1)}{p-2}})}\|y\|_{L^q(0,T;L^p)} \nonumber \\
& \lesssim T^{\theta}\|y\|^{\alpha-1}_{L^q(0,T;W^{2,p})}\|y\|_{L^q(0,T;W^{2,p})}, \\
\label{eq4}
\||y|^{\alpha-1}|\nabla y|\|_{L^{q'}(0,T;L^{p'})} & \le T^{\theta}\|y\|^{\alpha-1}_{L^q(0,T;L^{\frac{p(\alpha-1)}{p-2}})}\|\nabla y\|_{L^q(0,T;L^p)} \nonumber \\
& \lesssim T^{\theta}\|y\|^{\alpha-1}_{L^q(0,T;W^{2,p})}\|y\|_{L^q(0,T;W^{2,p})}, \\
\label{eq5}
\||y|^{\alpha-2}|\nabla y|^2\|_{L^{q'}(0,T;L^{p'})} & \le T^{\theta}\|y\|^{\alpha-2}_{L^q(0,T;L^{\frac{d(\alpha +1)}{d-4}})}\|\nabla y\|^2_{L^q(0,T;L^{\frac{d(\alpha+1)}{d+\alpha -3}})} \nonumber \\
& \lesssim T^{\theta}\|y\|^{\alpha-2}_{L^q(0,T;W^{2,p})}\|y\|^2_{L^q(0,T;W^{2,p})}, \\
\label{eq6}
\||y|^{\alpha-1}|\Delta y|\|_{L^{q'}(0,T;L^{p'})} & \le T^{\theta}\|y\|^{\alpha-1}_{L^q(0,T;L^{\frac{p(\alpha-1)}{p-2}})}\|\Delta y\|_{L^q(0,T;L^p)} \nonumber \\
& \lesssim T^{\theta}\|y\|^{\alpha-1}_{L^q(0,T;W^{2,p})}\|y\|_{L^q(0,T;W^{2,p})}.
\end{align}
Thus, inserting (\ref{eq3}),(\ref{eq4}),(\ref{eq5}),(\ref{eq6}) into (\ref{eq2}) and (\ref{eq1}) yields that 
\begin{align}
\label{W2p}
\|F(y)\|_{L^q(0,T;W^{2,p})}\lesssim \|x\|_{H^2}+T^{\theta}\|y\|^{\alpha}_{L^q(0,T;W^{2,p})}.
\end{align}

Fix $\omega \in \Omega$ and consider $F$ on the set
\[ \mathcal{Y}^{\tau_1}_{M_1}=\left\{ y\in L^{\infty}(0,\tau_1;H^2)\cap L^q(0,\tau_1;W^{2,p});\|y\|_{L^{\infty}(0,\tau_1;H^2)}+\|y\|_{L^q(0,\tau_1;W^{2,p})}\le M_1 \right\}. \]
Then, suggested to the proof of Theorem \ref{ymain}. We can conclude $F$ is a contraction from $L^{\infty}(0,\tau_1;L^2)\cap L^q(0,\tau_1;L^p)$ to the same space, which yields a unique $y\in \mathcal{Y}^{\tau_1}_{M_1}$ with $y=F(y)$ on $[0,\tau_1]$. We also have $y_1|_{[0,\tau_1]}\in L^{\gamma}(0,\tau_1;W^{2,\rho})$ by using Strichartz estimate (\ref{SEH^2}). The same is true for $n\ge 2$. Therefore, we have proved the case $d=5,6,7.$

If $d=4$, choose the Strichartz pair $(p,q)=(\frac{2(\alpha+2)}{\alpha+1},\alpha+2)$. Then H$\ddot{\text{o}}$lder's inequality and Sobolev's imbedding give
\begin{align}
\||y|^{\alpha-1}y\|_{L^{q'}(0,T;L^{p'})} & \le T^{\theta}\|y\|^{\alpha-1}_{L^q(0,T;L^{\frac{p(1-\alpha)}{2-p}})}\|y\|_{L^q(0,T;L^p)} \nonumber \\
& \lesssim T^{\theta}\|y\|^{\alpha-1}_{L^q(0,T;W^{2,p})}\|y\|_{L^q(0,T;W^{2,p})}, \\
\||y|^{\alpha-1}|\nabla y|\|_{L^{q'}(0,T;L^{p'})} & \le T^{\theta}\|y\|^{\alpha-1}_{L^q(0,T;L^{\frac{p(1-\alpha)}{2-p}})}\|\nabla y\|_{L^q(0,T;L^p)} \nonumber \\
& \lesssim T^{\theta}\|y\|^{\alpha-1}_{L^q(0,T;W^{2,p})}\|y\|_{L^q(0,T;W^{2,p})}, \\
\||y|^{\alpha-2}|\nabla y|^2\|_{L^{q'}(0,T;L^{p'})} & \le T^{\theta}\|y\|^{\alpha-2}_{L^q(0,T;L^{\infty})}\|\nabla y\|^2_{L^q(0,T;L^{\frac{4(\alpha+2)}{\alpha +3}})} \nonumber \\
& \lesssim T^{\theta}\|y\|^{\alpha-2}_{L^q(0,T;W^{2,p})}\|y\|^2_{L^q(0,T;W^{2,p})}, \\
\||y|^{\alpha-1}|\Delta y|\|_{L^{q'}(0,T;L^{p'})} & \le T^{\theta}\|y\|^{\alpha-1}_{L^q(0,T;L^{\frac{p(1-\alpha)}{2-p}})}\|\Delta y\|_{L^q(0,T;L^p)} \nonumber \\
& \lesssim T^{\theta}\|y\|^{\alpha-1}_{L^q(0,T;W^{2,p})}\|y\|_{L^q(0,T;W^{2,p})}.
\end{align}
Hence, the estimate (\ref{W2p}) is accordingly modified by
\begin{align}
\|F(y)\|_{L^q(0,T;W^{2,p})}\lesssim \|x\|_{H^2}+T^{\theta}\|y\|^{\alpha}_{L^q(0,T;W^{2,p})}.
\end{align}
Thus we can show the case $d=4$ in a similar way as for $d=5,6,7$.

If $d=1,2,3$, we can use Sobolev's embedding theorem $\|y\|_{L^{\infty}}\lesssim \|y\|_{H^2}$. Then, we have
\begin{align*}
\|F(y)\|_{L^{\infty}(0,T;H^2)} & \lesssim \|x\|_{H^2}+\|\lambda g(y)\|_{L^1(0,T;H^2)} \\
& \lesssim \|x\|_{H^2}+(\||y|^{\alpha-1}y\|_{L^1(0,T;L^2)}+\||y|^{\alpha-1}|\nabla y|\|_{L^1(0,T;L^2)} \\
& \qquad +\||y|^{\alpha-2}|\nabla y|^2\|_{L^1(0,T;L^2)}+\||y|^{\alpha-1}|\Delta y|\|_{L^1(0,T;L^2)}) \\
& \lesssim \|x\|_{H^2}+T\|y\|^{\alpha-1}_{L^{\infty}(0,T;L^{\infty})}\|y\|_{L^{\infty}(0,T;H^2)}+T\|y\|^{\alpha-2}_{L^{\infty}(0,T;L^{\infty})}\|y\|^2_{L^{\infty}(0,T;H^2)} \\
& \lesssim \|x\|_{H^2}+T\|y\|^{\alpha}_{L^{\infty}(0,T;H^2)}.
\end{align*}
Thus we can show the case $d=1,2,3$ in a similar way as for $d=5,6,7$.

Next, prove the continuous dependence of the initial data. Suppose that $x_m\to x$ in $H^2$ and let $(y_m,(\tau_n^m)_{n\in \mathbb{N}},\tau^*(x_m))$ be the unique local solutions of (\ref{RSNLS}) corresponding to the initial data $x_m, \ m\ge 1.$ Choose the Strichartz pair $(p,q)=(\frac{d(\alpha+1)}{d+2\alpha -2},\frac{4(\alpha+1)}{(d-4)(\alpha-1)}), \ \text{if} \ d=5,6,7, \ (p,q)=(\frac{2(\alpha+2)}{\alpha+1},\alpha+2), \ \text{if} \ d=4$. If $d=1,2,3$, take any Strichartz pair. Since $\|x_m\|_{H^2}\le \|x\|_{H^2}+1$ for $m\ge m_1$ with $m_1$ large enough, we modify $\tau_1(\le T)$ in the proof of Theorem \ref{ymain} by 
\[ \tau_1=\inf \left\{ t\in [0,T]:2\cdot3^{\alpha-1}C_t^{\alpha}(\|x\|_{H^2}+1)^{\alpha-1}t^{\theta}>\frac{1}{3} \right\} \wedge T, \]
such that $\tau_1$ is independent for $m\ge m_1$. Hence, 
\[ \widetilde{R}:=\sup_{m\ge m_1}\|y_m\|_{L^{\infty}(0,\tau_1;H^2)}<\infty, \quad \Pas. \]

We first prove the continuous dependence of the initial data on the interval $[0,\tau_1]$. From proof of Theorem 1.2 of \cite{BRZ16}, taking $t$ small and independent of $m(\ge m_1)$, we have
\begin{align}
\label{CDH^1}
\|y_m-y\|_{L^{\infty}(0,t;H^1)}+\|y_m-y\|_{L^q(0,t;W^{1,p})}\to 0, \ \text{as} \ m\to \infty.
\end{align}
Then, to obtain that
\begin{align}
\label{CDH^2}
\|y_m-y\|_{L^{\infty}(0,t;H^2)}+\|y_m-y\|_{L^q(0,t;W^{2,p})}\to 0,
\end{align}
we use (\ref{deltau}) to derive that for $m\ge m_1$
\begin{align}
\label{deltaym}
\Delta (y_m-y) & = U(t,0)\Delta (x_m-x)+\int_0^tU(t,s)\{ 2i(D_j\nabla \tilde{b}^j(s)+\nabla \tilde{b}^j(s)D_j+\nabla \tilde{c}(s)) \nonumber \\
& \quad \times \nabla(y_m-y)+i(D_j\Delta \tilde{b}^j(s)+\Delta \tilde{b}^j(s)D_j+\Delta \tilde{c}(s))(y_m-y) \nonumber \\
&  \quad -\lambda i\Delta (g(y_m(s))-g(y(s)))\}ds,
\end{align}
where $g(y)=|y|^{\alpha-1}y$. We note that, by Proposition 2.3(a) in \cite{MMT08}, (\ref{XEL^2}) and (\ref{XEH^1}), similar to the estimate in \eqref{gesti}, we obtain
\begin{align}
\label{eq426}
\|i(D_j\Delta \widetilde{b}^j+\Delta \widetilde{b}^jD_j+\Delta \widetilde{c})(y_m-y)\|_{\widetilde{X}'_{[0,t]}}&\lesssim \|y_m-y\|_{\widetilde{X}_{[0,t]}} \nonumber \\
&\lesssim \|x_m-x\|_{L^2}+\|g(y_m)-g(y)\|_{L^{q'}(0,t;L^{p'})} \nonumber \\
&\lesssim \|x_m-x\|_{L^2}+t^{\theta}\|y_m-y\|_{L^q(0,t;L^p)}, \\
\label{eq427}
\|2i(D_j\nabla \tilde{b}^j+\nabla \tilde{b}^jD_j+\nabla \tilde{c})\nabla (y_m-y)\|_{\tilde{X}'_{[0,t]}}&\lesssim \|\nabla(y_m-y)\|_{\widetilde{X}_{[0,t]}} \nonumber \\
&\lesssim \|x_m-x\|_{H^1}+\|g(y_m)-g(y)\|_{L^{q'}(0,t;W^{1,p'})} \nonumber \\
&\lesssim \|x_m-x\|_{H^1}+t^{\theta}\|y_m-y\|_{L^q(0,t;W^{1,p})},
\end{align}
where $\widetilde{X}_{[0,t]}$ is the local smoothing space (see Definition \ref{locsp}).  Applying (\ref{XEL^2}), (\ref{XEH^1}) to (\ref{deltaym}), we derive by (\ref{eq426}) and (\ref{eq427}), 
\begin{align}
\label{eq428}
&\|\Delta y_m-\Delta y\|_{L^{\infty}(0,t;L^2)}+\|\Delta y_m-\Delta y\|_{L^q(0,t;L^p)} \nonumber \\
& \lesssim \|\Delta x_m-\Delta x\|_{L^2}+\|2i(D_j\nabla \widetilde{b}^j+\nabla \widetilde{b}^jD_j+\nabla \widetilde{c})\nabla (y_m-y)\|_{\widetilde{X}'_{[0,t]}} \nonumber \\
& \quad +\|i(D_j\Delta \widetilde{b}^j+\Delta \widetilde{b}^jD_j+\Delta \widetilde{c})(y_m-y)\|_{\widetilde{X}'_{[0,t]}}+\|\lambda i\Delta (g(y_m)-g(y))\|_{L^{q'}(0,t;L^{p'})} \nonumber \\
& \lesssim \|x_m-x\|_{H^2}+t^{\theta}\|y_m-y\|_{L^q(0,t;W^{1,p})}+\|\Delta g(y_m)-\Delta g(y)\|_{L^{q'}(0,t;L^{p'})}.
\end{align}
As regards the last term on the right hand side of (\ref{eq428}), we note that $\nabla g(y)=F_1(y)\nabla y+F_2(y)\nabla \overline{y}$, where $F_1(y)=\frac{\alpha+1}{2}|y|^{\alpha-1}$ and $F_2(y)=\frac{\alpha-1}{2}|y|^{\alpha-3}y^2$. Then 
\begin{align*}
\Delta g(y_m)-\Delta g(y) & =\nabla F_1(y_m)\nabla y_m+F_1(y_m)\Delta y_m+\nabla F_2(y_m)\nabla \overline{y_m}+F_2(y_m)\Delta \overline{y_m} \\
&\quad -\nabla F_1(y)\nabla y-F_1(y)\Delta y-\nabla F_2(y)\nabla \overline{y}-F_2(y)\Delta \overline{y} \\
& = F_1(y_m)[\Delta y_m-\Delta y]+[F_1(y_m)-F_1(y)]\Delta y+F_2(y_m)[\Delta \overline{y_m}-\Delta \overline{y}] \\
& \quad +[F_2(y_m)-F_2(y)]\Delta \overline{y}+\nabla F_1(y_m)[\nabla y_m-\nabla y]+[\nabla F_1(y_m)-\nabla F_1(y)]\nabla y \\
& \quad +\nabla F_2(y_m)[\nabla \overline{y_m}-\nabla \overline{y}]+[\nabla F_2(y_m)-\nabla F_2(y)]\nabla \overline{y} \\
&=:\sum_{k=1}^8I_k.
\end{align*}
Since $|I_1|+|I_3| \lesssim |y_m|^{\alpha-1}|\Delta y_m-\Delta y|$, we have
\begin{align}
\label{412}
\|I_1+I_3\|_{L^{q'}(0,t;L^{p'})} &\lesssim \|y_m\|^{\alpha-1}_{L^{\infty}(0,t;L^r)}t^{\theta}\|\Delta y_m-\Delta y\|_{L^{q}(0,t;L^{p})} \nonumber \\
&\lesssim \widetilde{R}^{\alpha-1}t^{\theta}\|y_m-y\|_{L^{q}(0,t;W^{2,p})}.
\end{align}
Here we used H$\ddot{\text{o}}$lder's inequality with
\[ 1-\frac{1}{q}=\frac{\alpha-1}{\infty}+\theta+\frac{1}{q}, \quad 1-\frac{1}{p}=\frac{\alpha-1}{r}+\frac{1}{p}, \quad r=\frac{p(\alpha-1)}{p-2}. \]
Also, since $|I_5|+|I_7| \lesssim |y_m|^{\alpha-2}|\nabla y_m-\nabla y|$, we have
\begin{align}
\label{413}
\|I_5+I_7\|_{L^{q'}(0,t;L^{p'})} &\lesssim \|y_m\|^{\alpha-2}_{L^{\infty}(0,t;L^r)}t^{\theta}\|\nabla y_m-\nabla y\|_{L^{q}(0,t;L^{p})} \nonumber \\
&\lesssim \widetilde{R}^{\alpha-2}t^{\theta}\|y_m-y\|_{L^{q}(0,t;W^{2,p})}.
\end{align}
Here we used H$\ddot{\text{o}}$lder's inequality with
\[ 1-\frac{1}{q}=\frac{\alpha-2}{\infty}+\theta+\frac{1}{q}, \quad 1-\frac{1}{p}=\frac{\alpha-2}{r}+\frac{1}{p}, \quad r=\frac{p(\alpha-2)}{p-2}. \]
Therefore, by (\ref{eq428}),(\ref{412}) and (\ref{413}),
\begin{align}
\label{414}
&\|y_m-y\|_{L^{\infty}(0,t;H^2)}+\|y_m-y\|_{L^q(0,t;W^{2,p})} \nonumber \\
& \lesssim \|x_m-x\|_{H^2}+t^{\theta}\|y_m-y\|_{L^q(0,t;W^{1,p})}+\widetilde{R}^{\alpha-1}t^{\theta}\|y_m-y\|_{L^{q}(0,t;W^{2,p})} \nonumber \\
& \quad +\widetilde{R}^{\alpha-2}t^{\theta}\|y_m-y\|_{L^{q}(0,t;W^{2,p})}+\|I_2+I_4+I_6+I_8\|_{L^{q'}(0,t;L^{p'})}.
\end{align}
Also, similar to the proof of Theorem 1.2 of \cite{BRZ16}, it follows that 
\begin{align}
\label{415}
\|I_2\|_{L^{q'}(0,t;L^{p'})}+\|I_4\|_{L^{q'}(0,t;L^{p'})}+\|I_6\|_{L^{q'}(0,t;L^{p'})}+\|I_8\|_{L^{q'}(0,t;L^{p'})} \to 0, \ \text{as} \ m\to \infty.
\end{align}
Thus we obtain (\ref{CDH^2}) by (\ref{CDH^1}), (\ref{414}) and (\ref{415}).
Reiterating this procedure in finite steps we obtain (\ref{CDH^2}) on $[0,\tau_1]$. 

Now, since $y_m(\tau_1)\to y(\tau_1) \ \text{in} \ H^2$, similarly we can get the above results to $[\tau_1,\tau_2]$ with $\tau_2$ depending on $\|y(\tau_1)\|_{H^2}$. Therefore, we can show the continuous dependence on $[0,\tau_2]$. Reiterating this procedure, we then obtain increasing stopping times $\tau_n$, depending on $\|y(\tau_{n-1})\|_{H^2}$, such that continuous dependence holds on every $[0,\tau_n]$. Therefore, for $n\ge1$ and $\Pas \ \omega \in \Omega$, the map $x\to y(\cdot,x,\omega)$ is continuous from $H^2$ to $L^{\infty}(0,\tau_n;H^2)\cap L^q(0,\tau_n;W^{2,p})$. The same is true for any Strichartz pair $(\rho, \gamma)$. 

Finally, we prove the blowup alternative. Suppose that 
\[\Pro (M^*<\infty;\tau_n<\tau^*(x), \ \forall n\in \mathbb{N})>0,\] 
where
\[ M^*:=\sup_{t\in[0,\tau^*(x))}\|y(t)\|_{H^2}.\] 
Define 
\[ Z_t=2\cdot3^{\alpha-1}C_t^{\alpha}(M^*)^{\alpha-1}t^{\theta}, \ t\in [0,T], \]
\[ \sigma:=\inf \left\{ t\in [0,T]:Z_t>\frac{1}{3} \right\} \wedge T. \]
For $\omega \in \{ M^*<\infty;\tau_n<\tau^*(x), \ \forall n\in \mathbb{N} \}$, since $\tau_n(\omega)<T, \ \forall n\in \mathbb{N},$ by the definition of $\sigma_n$ in Step 2, we have 
\[ Z_t^{(n+1)}=2\cdot3^{\alpha-1}C_{\tau_n+t}^{\alpha}\|y(\tau_n)\|_{H^2}^{\alpha-1}t^{\theta}, \ t\in [0,T-\tau_n(\omega)], \]
\[ \sigma_n(\omega)=\inf \left\{ t\in[0,T-\tau_n(\omega)]: Z_t^{(n+1)}>\frac{1}{3} \right\}\wedge (T-\tau_n). \]

Notice that, for every $n\ge 1, \ \|y(\tau_n(\omega))\|_{H^2}\le M^*, \ C_{\tau_n(\omega)+t}\le C_{T+t}$. It follows that $Z_t(\omega)\ge Z_t^{(n+1)}(\omega)$, therefore $\sigma_n(\omega)>\sigma(\omega)>0.$ Hence $\tau_{n+1}(\omega)=\tau_n(\omega)+\sigma_n(\omega)>\tau_n(\omega)+\sigma(\omega)$, which implies $\tau_{n+1}(\omega)>\tau_1(\omega)+n\sigma(\omega)$ for every $n\ge 1$, contradicting the fact that $\tau_n(\omega)\le T$. This completes the proof.
\end{proof}

\section{Proof of Theorem \ref{yglob}}
\setcounter{equation}{0}

Under the assumption that the $H^1$ norm is bounded, we prove the time global well-posedness in $H^2$ for the equation with power condition $2\le \alpha \le 1+\frac{4}{(d-4)^+}$. This is so called persistence argument.
\begin{proof}[Proof of Theorem \ref{yglob}]
From Theorem \ref{ycor}, for any $x\in H^2$ and $0<T<\infty$, there exists a local solution $(y,\tau^*(x))$ where $\tau^*(x)$ is the maximum existence time of $y$. Assume that $\tau^*(x)<T$. Then, there exists $\varepsilon$ such that $\tau^*(x)-\varepsilon<\tau^*(x)+\varepsilon<T$. Set $I=[\tau^*(x)-\varepsilon,\tau^*(x)+\varepsilon]$. Fix $\omega\in\Omega$. Set $(p,q)=(\alpha+1,\frac{4(\alpha+1)}{d(\alpha-1)})$. We estimate the map in (\ref{Fy2}). From the deterministic Strichartz estimate, we have
\begin{align*}
\|F(y)\|_{L^{\infty}(I;H^2)\cap L^q(I;W^{2,p})}\le \|x\|_{H^2}+\||y|^{\alpha-1}y\|_{L^{q'}(I;W^{2,p'})}.
\end{align*}
Thus, by Sobolev's embedding theorem, we have
\begin{align*}
&\||y|^{\alpha-1}y\|_{L^{q'}(I;W^{2,p'})} \\
& \lesssim \||y|^{\alpha-1}y\|_{L^{q'}(I;L^{p'})}+\||y|^{\alpha-1}|\nabla y|\|_{L^{q'}(I;L^{p'})}+\||y|^{\alpha-2}|\nabla y|^2\|_{L^{q'}(I;L^{p'})}+\||y|^{\alpha-1}|\Delta y|\|_{L^{q'}(I;L^{p'})} \\
& \lesssim (2\varepsilon)^{\theta}\|y\|^{\alpha-1}_{L^{\infty}(I;L^p)}\|y\|_{L^q(I;L^p)}+(2\varepsilon)^{\theta}\|y\|^{\alpha-1}_{L^{\infty}(I;L^p)}\|\nabla y\|_{L^q(I;L^p)} \\
& \quad +(2\varepsilon)^{\theta}\|y\|^{\alpha-2}_{L^{\infty}(I;L^p)}\|\nabla y\|^2_{L^q(I;L^p)}+(2\varepsilon)^{\theta}\|y\|^{\alpha-1}_{L^{\infty}(I;L^p)}\|\Delta y\|_{L^q(I;L^p)} \\
& \lesssim (2\varepsilon)^{\theta}\|y\|^{\alpha-1}_{L^{\infty}(I;H^1)}\|y\|_{L^q(I;W^{2,p})}+(2\varepsilon)^{\theta}\|y\|^{\alpha-1}_{L^{\infty}(I;H^1)}\|y\|_{L^q(I;W^{2,p})} \\
& \quad +(2\varepsilon)^{\theta}\|y\|^{\alpha-2}_{L^{\infty}(I;H^1)}\|y\|^2_{L^q(I;W^{2,p})}+(2\varepsilon)^{\theta}\|y\|^{\alpha-1}_{L^{\infty}(I;H^1)}\|y\|_{L^q(I;W^{2,p})} \\
& \lesssim (2\varepsilon)^{\theta}\|y\|_{L^q(I;W^{2,p})}.
\end{align*}
In the similar way, we estimate the difference to obtain
\begin{align*}
\|F(y_1)-F(y_2)\|_{L^{\infty}(I;H^2)\cap L^q(I;W^{2,p})}\lesssim (2\varepsilon)^{\theta}C(T)\|y_1-y_2\|_{L^q(I;W^{2,p})}.
\end{align*}
Thus, if $\varepsilon$ is sufficiently small, $F$ is a contraction map. Therefore, it contradicts the definition of the maximal existence time $\tau^*(x)$.
\end{proof}

\section*{Acknowledgement}
The third author was supported by JSPS Kakenhi Grant-in-Aid for Scientific Research (C) 20K03671.

\section*{Declarations}
\subsection*{Conflicts of interests}
The authors declare that there is no conflict of interests regarding the publication of this paper.
\subsection*{Data Availability Statements}
Data sharing not applicable to this article as no datasets were generated or analysed during the current study.

\end{document}